\documentclass{amsart}
\usepackage{amssymb}
\usepackage{color}
\definecolor{darkgreen}{cmyk}{1,0,1,.2}
\definecolor{m}{rgb}{1,0.1,1}
\definecolor{green}{cmyk}{1,0,1,0}
\definecolor{test}{rgb}{1,0,0}   
\definecolor{cmyk}{cmyk}{0,1,1,0}

\newtheorem{Equation}{}[section]
\newtheorem{theorem}[Equation]{Theorem}
\newtheorem{proposition}[Equation]{Proposition}
\newtheorem{lemma}[Equation]{Lemma}
\newtheorem{corollary}[Equation]{Corollary}

\newtheorem{definition}[Equation]{Definition}

\newtheorem{remark}[Equation]{Remark}

\def\pa{\partial}

\def\Hom{\operatorname{Hom}}
\def\max{\operatorname{max}}
\def\JLO{\operatorname{JLO}}
\def\SF{\operatorname{SF}}
\def\Ch{\operatorname{Ch}}
\def\End{\operatorname{End}}
\def\Ker{\operatorname{Ker}}
\def\Ind{\operatorname{Ind}}

\def\ch{\operatorname{ch}}

\def\B{\mathbb B}
\def\A{\mathbb A}
\def\C{\mathbb C}
\def\D{\mathbb D}
\def\R{\mathbb R}
\def\Z{\mathbb Z}
\def\N{\mathbb N}

\def\ep{\epsilon}

\def\maA{{\mathcal A}}
\def\maC{{\mathcal C}}
\def\maB{{\mathcal B}}

\def\maE{{\mathcal E}}
\def\maJ{{\mathcal J}}
\def\maL{{\mathcal L}}

  \input xypic
  \xyoption{all}
\usepackage{color}
\definecolor{darkgreen}{cmyk}{1,0,1,.2}
\definecolor{m}{rgb}{1,0.1,1}

\begin{document}



\title[JLO and spectral flow for families]
{Higher spectral flow and \\ an entire bivariant JLO cocycle
}


\author[M-T. Benameur]{Moulay-Tahar Benameur}
\address{UMR 7122, LMAM, \\Universit\'e Paul Verlaine-Metz,
FRANCE}
\email{benameur@univ-metz.fr}
\author[A. L. Carey \today]{Alan L.  Carey}
\address{Mathematical Sciences Institute\\
Australian National University\\
Canberra, ACT. 0200, AUSTRALIA} 
\email{acarey@maths.anu.edu.au}
\footnote{Both authors acknowledge the financial support of the Australian Research Council and of the PICS,
Progr\`es en Analyse G\'eom\'etrique et Applications of the CNRS.}

\begin{abstract} For a single Dirac operator on a closed manifold the cocycle introduced by Jaffe-Lesniewski-Osterwalder \cite{JLO} (abbreviated here to JLO), is a representative
of Connes' Chern character map from the K-theory of the  algebra of smooth functions on the manifold to its entire cyclic cohomology.
Given a smooth fibration of closed manifolds and a family of generalized Dirac operators along the fibers, we define in this paper an associated bivariant JLO cocycle. We then prove that, for any $\ell \geq 0$, our bivariant JLO cocycle is entire  when we endow smoooth functions on the total manifold with the  $C^{\ell+1}$  topology and functions on the base manifold with the $C^\ell$  topology.  As a by-product of our theorem, we deduce that the bivariant JLO cocycle is entire for the Fr\'echet smooth topologies. We then prove that our JLO bivariant cocycle computes the Chern character of the Dai-Zhang higher spectral flow.
\end{abstract}
\maketitle
\tableofcontents

\section*{Introduction} 

%
%
%
%

Our objective in this paper is to give a bivariant entire Chern character sufficiently general  to encompass the  index theorem for families of
generalized Dirac operators.
In other words we make explicit the long held view that 
the Bismut formalism may be incorporated into noncommutative geometry.
 From the authors' point of view this question arose from a discussion at
Oberwolfach (we thank Masoud Khalkhali and Alain Connes for
comments). A number of  results in different algebraic and/or geometric situations have
been obtained previously for instance in \cite{LP, Nistor, Wu97, Gorokhovsky, Perrot}. 
(There is also the related question of the bivariant version of the
Connes-Moscovici residue cocycle but we defer that to another place.)

The first issue  is to choose a bivariant framework for
this problem.   In Meyer's thesis
\cite{MeyerThesis} we have found an appropriate formalism.
Using Meyer's ideas we define a  bivariant JLO (Jaffe-Lesniewski-Osterwalder \cite{JLO}) cocycle that encompasses the local  index theorem for families \cite {BGV, Bismut}. {Recall that
the ordinary JLO cocycle is a representative of Connes' Chern character map from the K-theory of an algebra to its entire cyclic cohomology \cite{ConnesBook}. As explained in \cite{BF} it may be used to deduce the Atiyah-Singer local index theorem. Our bivariant JLO cocycle generalizes this result to the
situation of families of generalized Dirac operators associated to a fibration. 

There are several intermediate results that are required to reach our objective.}
Our first main result is to prove that our bivariant JLO cocycle is entire in the sense of 
our adaptation of the formalism of \cite{MeyerThesis}, when one considers the $C^{\ell}$ topology on the algebra of smooth functions on the base, and the $C^{\ell+1}$ topology on the algebra of smooth functions on the total space of the fibration. We thank Ralf Meyer for his helpful comments on our approach to this result.

  To describe our further results we need some notation.
 Throughout we shall consider a locally trivial fibration $F\to M \stackrel{\pi}{\rightarrow} B$ of closed manifolds endowed with smooth metrics. Following \cite{BGV}, we 
fix an hermitian vector bundle $E\to M$ whose fibers are modules over the Clifford algebra of the fiberwise tangent bundle $T_v M = \Ker( \pi_*)$.  
While we formulate our initial results in a general way our main interest lies
in the case of odd dimensional fibers for $E$ as this is not as well understood
as is the even dimensional situation. 
We follow \cite{Bismut, BGV} and introduce on $E$ a quasiconnection
$\nabla$. We choose a family of generalized Dirac operators parametrised smoothly by $B$ denoted  $D$, and different superconnections as in \cite{Bismut}. They will be for us given by
$$
\A_\sigma (A) := \B_\sigma + A,
$$
where $A$ is a zero-th order fiberwise $\sigma$-pseudodifferential operator with coefficients in differential forms of positive degree (in applications $\geq 2$) on $B$. 
Here, the superconnection $\B_\sigma$ is defined by 
$$
\B_\sigma = \nabla + \sigma D
$$
with $\sigma$ being a Clifford variable as in Quillen's work \cite{Quillen88}.
We focus on  $\A_\sigma (A)$, $\B_\sigma$ or on superconnections  obtained by metric rescalings. In order to simplify the exposition of our results, we restrict ourselves to the case $\A(0) = \B_\sigma$ and  briefly explain later how the proofs extend to the general case. 

Our main application is to connect
our bivariant Chern character with the Dai-Zhang higher spectral flow.
The method we use draws on some ideas in
\cite{GetzlerOdd}, \cite{CPII} and \cite{Quillen88} on superconnections and the JLO cocycle.
We show using \cite{DaiZhang} that the bivariant JLO  Chern character in the case of the Bismut superconnection
gives the Bismut local formula that Dai-Zhang observe computes higher spectral flow. (The use of the Bismut superconnection avoids the complications that enter into our bivariant formula when $\nabla^2\neq 0$.) 

Thus our approach
explains the Dai-Zhang results in terms of bivariant cyclic cohomology.  In the single operator case these results are known from Connes \cite{ConnesJLO} and Getzler \cite{GetzlerOdd}. We choose here  to extend Getzler's method to deal with families. In \cite{Perrot}, similar issues are discussed for interesting noncommutative bivariant situations, but only encompassing in the commutative case trivial fibrations
and flat quasiconnections $\nabla$. Our motivation for giving a detailed treatment of the general commutative case has to do with
some applications in noncommutative settings, see for instance \cite{BenameurGorokhovsky},
\cite{BHI}.
We are aware of a need to use our point of view and results
in current work of  colleagues 
A. Gorokhovsky, J.-M. Lescure and B.-L. Wang. Thus we have written the exposition so that it
will adapt immediately to foliations and to spectral flow for 
twisted families. 

The  conceptual framework for our results is captured by
 the following commutative square.

\vspace{0.1in}
\begin{picture}(415,90) 

\put(120,75){$K^1(M)$}
\put(125,65){ $\vector(0,-1){35}$}
\put(110,48){$\Ch $}

\put(185,85){$\SF(D, \cdot)$}
\put(165,79){$\vector(1,0){65}$}

\put(185,25){$\JLO(D)$}
\put(170,19){$\vector(1,0){60}$}

\put(245,75){$K^0(B)$}
\put(255,65){ $\vector(0,-1){35}$}
\put(270,48){$\ch$}

\put(95,15){$HE_1 (C^\infty(M))$}

\put(235,15){$H^{even} (B, \C)$}
\end{picture} 

We digress to explain the notation in this diagram.
The left vertical arrow represents the entire Chern character, while the right vertical arrow represents the usual Chern character with appropriate normalizations.
Our family of generalized Dirac operators
parametrised by $B$ defines an element $[D]$ of
$KK^1(M,B)$ and  the top horizontal arrow, which is the higher spectral flow map defined in \cite{DaiZhang} when the $K^1$ class of $D$ is trivial, is still well defined in general as the Kasparov product $\cap [D]$ by $[D]$.
Finally $\JLO(D)$ is our entire bivariant Chern character
which takes  values in even de Rham cohomology of $B$. 
 
\subsection*{Statement of the main results}
The bivariant JLO cochain is defined by the sequence $(\psi_n)_n$ 
of functionals which, for $(f_0, \cdots, f_n) \in C^\infty(M)^{n+1}$, are given by the formula
$$
 \psi_n (f_0, \cdots, f_n) := \left< \left< f_0, [\B_\sigma , f_1],  \cdots ,[\B_\sigma , f_n] \right> \right>_{\B_\sigma}\in \Omega^*(B),
$$
where the multilinear functional on the right hand side
is a generalized JLO functional whose precise form is explained at the beginning of Section \ref{Multi}. 
Our main results are then as follows.

\medskip

\noindent{\bf Theorem \ref{JLO}}. {\em We assume that the fibers of our fibration are odd dimensional manifolds.
For any $\ell\geq 0$, $\psi=(\psi_{2n+1})_{n\geq 0}$ is an $\ell$-entire bivariant cocycle  in the sense made precise  in Definition \ref{entiremorph}. 
This means that
$\psi$ is a bounded morphism from the entire bornological completion of the universal differential algebra of $C^{\infty}(M)$ to  the algebra of smooth differential forms on $B$
endowed with the natural bornology.}
\medskip

\noindent{\bf Theorem \ref{JLO=SF}}. {\em Assume that the index class of $D$ in $K^1(B)$ is trivial so that the higher spectral  flow $\SF(D,U)$ is well defined, for any $U\in GL_N(C^\infty (M))$. Then the following relation holds in the even de Rham cohomology of the base manifold $B$:
$$
\frac{1}{{\sqrt\pi}}<\JLO (D) , \ch(U) >  = \ch ( \SF (D,U)).
$$\em}

It is possible to deduce from our computations a more precise statement on forms rather than classes. 
 We see also as a corollary of our arguments, and of the main result of \cite{DaiZhang}, that $\JLO(D)$ coincides with the topological map $H^{odd}(M) \to H^{even} (B)$ given by
$$
\omega \longmapsto \int_{M/B} \omega\wedge {\hat A} (TM|B).
$$

To understand the structure of our exposition we now expand on the List of Contents. Section 1 gives the differential geometric framework: families of generalized Dirac operators, connections and
bivariant functionals. In Section 2 we introduce our bivariant JLO multilinear functionals and discuss the identities they satisfy
essentially following \cite{GetzlerSzenes} but adapted to families.
Our first objective, to understand the entire property for bivariant JLO, begins in Section 3. We summarise Meyer's point of view at the beginning of this Section for the reader's convenience and then state our main theorem.
Section 4 contains the proof: the argument is a series of estimates that establish that our JLO is entire in the sense of Meyer (and hence entire).
In Section 5 we establish the commutativity of the diagram above.

{\em Acknowledgements.} The authors are very grateful for a careful reading by a referee who highlighted some confusions and gaps in the original  exposition. We have also benefitted from discussions
with our colleagues, A. Connes, J. Cuntz, A. Gorokhovsky, J. Heitsch, M. Khalkahli, R. Meyer, D. Perrot, M. Puschnigg, A. Rennie, G. Yu,  to whom we express our appreciation. This work was progressed while the first author was visiting the MSI in the Australian National University, the second author was visiting the Laboratoire de Math\'ematiques et Applications in the Universit\'e Paul Verlaine-Metz, and 
both authors were visiting the  Mathematisches Forschungsinstitut Oberwolfach and the Hausdorff Institute for Mathematics.  Both authors are most grateful for the 
warm hospitality and generous support of their hosts.

\section{Preliminary results}

\subsection{Connections}\label{Prelim}

We begin by introducing some further notation and a general framework for our discussion.
As above we denote by $T_vM$ the fiberwise tangent bundle $T_vM := \Ker (\pi_*) \subset TM$. Then we are assuming there is a Clifford homomorphism of algebra bundles
$$
c: Cl(T_vM\otimes \C) \longrightarrow \End (E) \text{ with } c(\xi)^2 = |\xi|^2 \text{ for } \xi\in T_vM,
$$
where $Cl(T_vM\otimes \C)$ denotes the Clifford algebra bundle associated with the hermitian bundle $T_vM\otimes \C$, and $\End (E)$ is the algebra bundle of endomorphisms of $E$. 
We assume as usual that $E$ is endowed with a Clifford connection $\nabla^E$ and consider the Dirac operator $D$ associated with this connection. Then $D$ is a fiberwise first order differential operator acting on smooth sections of $E$ and can be regarded as a family of elliptic operators along the fibers smoothly parametrized by the elements of the base manifold $B$, i.e. $D = (D_b)_{b\in B}$ where  $D_b : C^\infty (M_b, E|_{M_b}) \to C^\infty (M_b, E|_{M_b})$ is an essentially self-adjoint generalized Dirac operator. Notice that we are using self-adjoint operators while the authors of \cite{BGV} use skew-adjoint operators.

Using the metric, we fix the horizontal distribution $H = (T_vM)^\perp$ so that {{
$$
TM = H \oplus T_vM \text{ and for any } b\in B, \pi_{*,m} : H_m \rightarrow T_{\pi(m)} B \text{ is a linear isomorphism}.
$$ }}
The dual vector bundle $\pi^*T^*B$ can be identified with the subbundle of $T^*M$ consisting of forms that vanish on vertical vectors. Using the splitting given by $H$, we deduce the existence of a restriction projection 
$\varrho: T^*M \to \pi^*T^*B$. This projection extends to exterior powers and we obtain, in  the obvious notation,
$$
\varrho: C^\infty (M, E\otimes \Lambda T^*M) \rightarrow C^\infty (M, E\otimes \Lambda \pi^*T^*B).
$$
Notice that $C^\infty (M, E\otimes \Lambda \pi^*T^*B)$ is a module  over the algebra of differential forms on the base manifold $B$. More precisely, this module structure is obtained using  the pull-back map associated with the projection $\pi$.  So if we denote by $\xi\omega $ the action of a differential form $\omega$ on $B$, on a smooth section $\xi$ of $E\otimes \Lambda \pi^*T^*B$, then 
$$
(\xi \omega) (m) := \xi(m) \wedge \omega (\pi(m)), \text{ or } \xi \omega := \xi \wedge \pi^*\omega.
$$
Here $\wedge$ is denoting the usual action of differential forms on the exterior algebra.

For any $m\in \Z$, we  denote by $\Psi^m(M | B, E)$ the space of ($1$-step polyhomogeneous) classical pseudodifferential operators of order $m$, acting along the fibers of $\pi: M\to B$, see \cite{AtiyahSinger4}. The local coefficients of such operators are thus smooth in the base variables and the space $\Psi^m(M | B; E)$ is a module over the algebra $C^\infty (B)$ of smooth functions on $B$. We shall also need the space of such operators with coefficients in differential forms on the base. More precisely, we set
$$
\psi^m (M|B,E; \Lambda^* B) = \Psi^m(M | B, E) {\widehat\otimes}_{C^\infty(B)} \Omega^* (B).
$$
According to \cite{AtiyahSinger4}, the space $\Psi^m(M | B, E; \Lambda^* B)$ can be endowed with a complete smooth topology as a projective tensor product of such topological spaces.
An element of $\psi^m (M|B,E; 
\Lambda^* B)$ is thus equivariant for the action of the algebra $\Omega^*(B)$ of smooth differential forms on $B$.
As a consequence  we can compose an element $P$ of $\psi^m (M|B,E; \Lambda^k B)$ with an element $Q$ of $\psi^{m'} (M|B,E; \Lambda^h B)$ to get an element $QP$ of $\psi^{m+m'} (M|B,E; \Lambda^{k+h} B)$. 
We shall denote by $\psi^\infty (M|B,E; \Lambda^* B)$ the algebra obtained in this way, i.e.
$$
\psi^\infty (M|B,E; \Lambda^* B) :=  \bigcup_{m\in \Z} \psi^m (M|B,E; \Lambda^* B). 
$$
As we shall see, the order of the pseudodifferential operators will not be involved in the gradings used in the sequel. In particular, for us, $\psi^\infty (M|B,E; \Lambda^* B)$ is $\Z_2$-graded by the parity of the degree of the forms on $B$. We also denote by $\psi^{-\infty} (M|B,E; \Lambda^* B)$ the ideal of fiberwise smoothing operators with coefficients in differential forms on $B$, i.e.
$$
\psi^{-\infty} (M|B,E; \Lambda^* B) := \bigcap_{m\in \Z} \psi^m (M|B,E; \Lambda^* B).
$$
Introduce the fiber product $G=M\times_B M:=\{(m,m')\in M\times M, \pi(m)= \pi(m')\}$, which is a smooth groupoid.
By the fiberwise Schwartz theorem applied to fiberwise smoothing operators, $\psi^{-\infty} (M|B,E; \Lambda^* B)$ can and will be identified with the convolution algebra of smooth sections of the bundle $\Hom(E)\otimes\Lambda^*T^*B$ over $G$ whose fiber is
$$
(\Hom(E)\otimes\Lambda^*T^*B)_{m,m'}:=\Hom(E_{m'}, E_m)\otimes\Lambda^*T_{\pi(m)=\pi(m')}^*B.
$$
\begin{definition}
We define the operator $\nabla$ by
$$
\nabla:=\varrho \circ \nabla^E : C^\infty (M, E\otimes \Lambda \pi^*T^*B) \rightarrow C^\infty (M, E\otimes \Lambda T^*M) \rightarrow C^\infty (M, E\otimes \Lambda \pi^*T^*B),
$$
Then $\nabla$  is called a quasi-connection. 
\end{definition}

\begin{remark}
In Bismut's viewpoint  $\nabla$ is a connection on an infinite dimensional Fr\'echet bundle on $B$.
\end{remark}
\begin{lemma}
The quasi-connection $\nabla$ increases the degree of the forms by one and satisfies the Leibniz rule
$$
\nabla (\xi \omega) = (-1)^{\pa \xi} \xi d_B\omega +  (\nabla \xi) \omega,
$$
where $\pa \xi$ is the form degree of the section $\xi$. Moreover, the curvature operator $\nabla^2$ of $\nabla$ is a fiberwise first order differential operator with coefficients in $2$-forms on the base manifold $B$. 
\end{lemma}

\begin{proof}
This argument is analogous to one introduced in  \cite{BHI}. As $\nabla^E$ is a connection and choosing the module action on the right, we have:
\begin{eqnarray*}
\nabla^E (\xi \omega) & = & (\nabla^E \xi)\wedge \pi^*\omega + (-1)^{\pa \xi} \xi \wedge d\pi^*\omega \\
& = & (\nabla^E \xi)\omega + (-1)^{\pa \xi} \xi \wedge \pi^*d_B\omega \\
& = & (\nabla^E \xi)\omega + (-1)^{\pa \xi} \xi d_B\omega.
\end{eqnarray*}
Therefore, applying $\varrho$, we obtain
$$
\nabla  (\xi \omega)  =  \varrho(\nabla^E \xi) \omega + (-1)^{\pa \xi} \xi d_B\omega  =  (\nabla \xi)\omega + (-1)^{\pa \xi} \xi d_B\omega. 
$$
Therefore, we have by the classical computation $\nabla^2 (\xi \omega) = (\nabla^2\xi)\omega$, i.e. $\nabla^2$ is $\Omega^*(B)$ linear, and hence it is a fiberwise differential operator with coefficients in $2$-forms on the base. By reducing to local cordinates on the base manifold $B$, one computes $\nabla^2$ using the local expression of $\nabla^E$. It is then easy to check that $\nabla^2$ is indeed a (fiberwise) first order differential operator, see also \cite{BHI} for more details.
\end{proof}

We shall say that an $\Omega^*(B)$-linear map on $C^\infty (M; E\otimes \pi^*\Lambda^*B)$ has degree $k\in \Z$, if it increases the form degree of the sections by $k$. Hence, such a map sends $C^\infty (M; E\otimes \pi^*\Lambda^hB)$ into $C^\infty (M; E\otimes \pi^*\Lambda^{h+k}B)$. We then set $\pa T = k$ and denote by $\maC^k$ the filtration obtained in this way, that is $\maC^k$ is the space of such maps with degree $\leq k$, which are $\Omega^*(B)$-linear operators.
\begin{lemma}
Let $T$ be an $\Omega^*(B)$-linear map of degree $k$, acting on $C^\infty (M; E\otimes \pi^*\Lambda^*B)$. Then the graded commutator 
$$
\pa (T) = [\nabla, T] := \nabla \circ T - (-1)^k T \circ \nabla,
$$
is an $\Omega^*(B)$-linear map which belongs to $\maC^{k+1}$. 
\end{lemma}

\begin{proof}
Fix $T\in \maC^k$. For $\xi\in C^\infty (M; E\otimes \pi^*\Lambda^hB)$ and $\omega\in \Omega^*(B)$, we can write
\begin{eqnarray*}
\pa (T) (\xi \omega ) & = & \nabla( (T\xi)\omega) - (-1)^k T(\nabla\xi)\omega - (1)^{k+h} (T\xi) d_B\omega \\
& = & \nabla (T\xi) \omega + (-1)^{h+k} (T\xi) d_B\omega -(-1)^k T(\nabla\xi)\omega - (1)^{k+h} (T\xi) d_B\omega \\
& = & (\pa (T) \xi) \omega.
\end{eqnarray*}
\end{proof}

\begin{proposition}
The derivation $\pa: \maC^* \to \maC^*$ preserves the subspace $\psi^{-\infty} (M|B,E; \Lambda^* B)$. More precisely, the derivation $\pa$ preserves each $\psi^{h} (M|B,E; \Lambda^* B)$ for  $h\in \Z$.
\end{proposition}

\begin{proof}
A classical computation shows that, in local coordinates over a small open set $V$ of $M$, the quasi-connection $\nabla$ has the expression
$$
\nabla  = d^\nu \otimes I_N+ \omega, \quad\text{ with } \omega\in M_N (C^\infty (V, \pi^*(T^*B))),
$$
where we have used a vector bundle isomorphism $E|_V \to V \times \C^N$ and  $d^\nu$ denotes the transverse de Rham derivative (in the direction $\pi^*(\Lambda T^*B)$) given with obvious notations by 
$
d^\nu = \varrho \circ d.
$
Taking commutators with the $0$-th order term $\omega$ clearly preserves $\psi^{h} (V|B,E; \Lambda^* B)$, since $\omega$ belongs to $\psi^{0} (V|B,E; \Lambda^1 B)$. Using a trivialization 
$$
(x;b)=(x_1, \cdots, x_p; b_1, \cdots, b_q) : V \longrightarrow \R^p \times \R^q.
$$
of the fibration $V\to U$ and the open set $U\subset B$,
one finds that there exists $A \in M_{q,p} (C_c^\infty (\R^n))$ such that
$$
d^\nu (f) = \sum_{i=1}^q [\frac{\pa f}{\pa b_i} + \sum_{j=1}^p A_{ij} \frac{\pa f}{\pa x_j}] db_i.
$$ 
The expression $\sum_{i=1}^q  db_i \sum_{j=1}^p A_{ij} \frac{\pa }{\pa x_j}$  is an element of $\psi^{1} (V|B; \Lambda^1 B)$,  a scalar operator. Therefore, when tensored with the identity of $\C^N$ it has a diagonal matrix as fiberwise principal symbol. 
Such a diagonal matrix graded commutes with the principal symbol of any element of $\psi^{h} (V|B,E; \Lambda^* B)$. So, the commutator with this operator preserves the order of the pseudodifferential operators. It thus remains to compute the commutator of  $d_B=\sum_{i=1}^q db_i \frac{\pa }{\pa b_i}$ with an element of $\psi^{h} (V|B,E; \Lambda^* B)$. Since the elements $P$ of $\psi^{h} (M|B,E; \Lambda^* B)$ are $\Omega^*(B)$-linear, we can restrict to $P\in \psi^{h} (M|B,E)$ and by pseudolocality, we can even assume that $P\in \psi^{h} (V|B,V\times \C^N)$ is given in the local coordinates $(x,b)$ by
$$
P (f) (x;b) = \frac{1}{(2\pi)^{p/2}}\int_{V_b \times \R^p} a (x;b; \xi) e^{i(x-x')\xi} f(x';b) dx' d\xi,
$$
where $f\in C_c^\infty (V, \C^N)$ and $a$ is the local total symbol of the classical fiberwise pseudodifferential operator $P$, an $N\times N$ matrix. A simple computation shows that the commutator $[d_B \otimes I_N, P]$ is given by the same local formula but with $a$ replaced by the matrix $d_B (a)$. Since this latter matrix is a classical symbol of order $h$, the proof is complete.
\end{proof}

\subsection{Bivariant cochains}

As before we are  considering  a smooth locally trivial fibration $\pi:M\to B$ of closed manifolds, together with the fiberwise generalized Dirac  operator $D$ acting on the smooth sections of the Clifford bundle $E$ over $M$. The dimension of the fibers is denoted by $p$ and the dimension of the base is $p'$. Recall that $\Omega^*(B, E)$ is the space of smooth sections over $M$ of the bundle $E\otimes \pi^*\Lambda^*T^*B$. This is a graded module over the graded Grassmann algebra of differential forms on the base manifold $B$. We defined in Section \ref{Prelim} the quasi-connection $\nabla$ associated with a connection on $E$ and the choice of a horizontal distribution $H$. 

Our bivariant $(n,k)$-cochains are linear maps $f: \maB_n = \maA^{\otimes_n} \to L_k$ where $L=\oplus_k L_k$ is a graded vector space (or a graded module over some algebra) or a graded algebra that will often be endowed with a `connection'. More precisely, we are interested in the graded spaces
$$
L= \Omega^*(B;E), \  \ L= \Omega^*(B), \ \ L_{-\infty}= \psi^{-\infty} (M|B, E; \Lambda^*B)\  \text{ and }\  L= \maL(\maE).
$$
Here $\psi^{-\infty} (M|B, E; \Lambda^*B)$ is the algebra of fiberwise smoothing operators with coefficients in forms. The space $\maE$ is the Hilbert module over the $C^*$-algebra $C(\Lambda^*T^*B)$, of continuous sections {{of $\Lambda^*T^*B$ over $B$, which is the completion of the smooth sections over $M$ of the bundle $E\otimes \pi^*\Lambda^*T^*B$.}} (For background on Hilbert C$^*$ modules see \cite{Lance}.) We choose connections on these  graded spaces (except for the last one) $$
\nabla, d_B=\text{de Rham differential} \text{ and } [\nabla, \cdot]
$$
respectively.
Note that we use the convention that all the commutators are graded ones. The space $\Hom(\maB,L)$ will then be bi-graded and we shall use the total grading for the commutators.  For any $P\in \psi^{h} (M|B, E; \Lambda^*B)$ and any $b\in B$, the operator $P_b$ belongs to $\psi^{h} (M_b, E|_{M_b}) \otimes \Lambda^*(T_b^*B)$. Therefore, for $h\leq 0$, $P_b$ extends to a $\Lambda^*(T_b^*B)$-linear bounded operator of the Hilbert space $L^2(M_b, E|_{M_b}) \otimes \Lambda^*(T_b^*B)$, hence an element of 
$$
B( L^2(M_b, E|_{M_b})) \otimes \Lambda^*(T_b^*B).
$$
Next, using a basis of the finite dimensional vector space $\Lambda^*(T_b^*B))$, we define a $\Lambda^*(T_b^*B)$-valued graded trace 
$$
\tau : L^1( L^2(M_b, E|_{M_b})) \otimes \Lambda^*(T_b^*B).\longrightarrow \Lambda^*(T_b^*B),
$$
where $L^1( L^2(M_b, E|_{M_b}))$ is the usual ideal of trace class operators on the Hilbert space $L^2(M_b, E|_{M_b})$. The expression ``graded trace'' means that $\tau$ vanishes on graded commutators, the grading being produced by the degree of the forms in $\Lambda^*(T_b^*B)$. The classical theory of pseudodifferential operators shows that $\psi^{h} (M_b, E|_{M_b}) \subset L^1( L^2(M_b, E|_{M_b}))$, for $h < -p$. Putting these traces together, we inherit a graded trace 
$$
\tau: L_{h}= \psi^{h} (M|B, E; \Lambda^*B)\longrightarrow \Omega^*(B) \text{ for any } h<-p.
$$
We shall for simplicity restrict $\tau$ to fiberwise smoothing operators and only consider
$$
\tau: L_{-\infty}= \psi^{-\infty} (M|B, E; \Lambda^*B)\longrightarrow \Omega^*(B).
$$
\begin{lemma}
The graded trace $\tau$ is closed, i.e. it satisfies the relation $\tau \circ \pa = d_B \circ \tau$.
\end{lemma}

\begin{proof}
({\it cf} \cite{BHI}). Fix a $P\in \psi^{-\infty} (M, E\otimes \Lambda^*(T^*B))$. Using the module structure, we can assume that  $P\in \psi^{-\infty} (M, E)$. Notice that each such $P$ has a smooth Schwartz kernel $k_P$. In local coordinates, $\nabla= d^\nu + M_\omega$ with $d^\nu= \varrho\circ d$ the de Rham derivative in the direction $\pi^*(\Lambda T^*B)$ and $M_\omega$ a zero-th order differential operator with coefficients in $1$-forms on $B$.  Clearly, the trace of the commutator $[M_\omega, P]$ is the integral {{of}} the trace of a commutator and hence is trivial. Using compactly supported smooth cut-off functions, we can assume that the smooth kernel $k_P $ is supported within a trivial open set diffeomorphic to $U\times U' \times W$ where $U$ and $U'$ are trivializing open sets in the typical fiber manifold $F$ and $W$ is a trivializing open set in the base $B$. The operator $d^\nu$ is given in the local coordinates $(x_1, \cdots, x_p; b_1, \cdots, b_q)$ of $U\times W$, by
$$
d^\nu = d_B + \sum_{i=1}^q db_i \sum_{j=1}^p A_{ij} \frac{\pa}{\pa x_j}\text{ with } A\in M_{q,p} (C_c^\infty (U\times W)) \text{ and } d_B= \sum_{i=1}^q db_i \frac{\pa}{\pa b_i}.
$$
We observe that
$$
\tau ([d_B, P])(b) = \int_{M_b} (d_B k_P) (b;x,x) dx = (d_B \circ \tau) (P) (b).
$$
It thus remains to show that $\tau ([\frac{\pa}{\pa x_j} , P]) = 0$. But since $P$ is fiberwise smoothing and $\frac{\pa}{\pa x_j}$ is a fiberwise first order differential operator, the proof is complete.
\end{proof}

 We follow \cite{Quillen88} and introduce an extra Clifford variable $\sigma$ of degree $1$ and central in the graded sense (it graded commutes with all operators). Hence we replace the algebra $L_{-\infty}$ by $L_{-\infty} [\sigma]$. We assume  that the fibers of our fibration are odd dimensional, so $p$ is odd. Then we extend the graded closed trace $\tau$ so that
$$
\pa: L_{-\infty}[\sigma] \rightarrow L_{-\infty}[\sigma] \text{ and } \tau_\sigma : L_{-\infty}[\sigma] \longrightarrow L=\Omega^*(B),
$$
by setting $\tau_\sigma (T + \sigma S) :=  \tau (S).$
We may also consider the algebra $\psi^\infty (M,E;\Lambda^*B)[\sigma]$ in the sequel. The total degree of an element $A$ in one of these extensions then takes into account  $\sigma$ and will  be denoted $|A|$.

\begin{lemma}\label{CommNabla}
The map $\tau_\sigma$ is a graded trace on the Clifford extension $L_{-\infty}[\sigma]$, with values in the graded algebra $\Omega^*(B)$. Moreover, it satisfies the relation
$
\tau_\sigma\circ \pa + d_B\circ \tau_\sigma = 0.
$
\end{lemma}

\begin{proof}
Let $A=T+\sigma S$ and $A'=T'+\sigma S'$ be elements of $L_{-\infty}[\sigma]$ with degrees  $k$ and $k'$ respectively. This means that $T$ and $T'$ have respectively degrees $k$ and $k'$ while $S$ and $S'$ have respectively degrees $(k-1)$ and $(k'-1)$. We then compute
\begin{eqnarray*}
 \tau_\sigma (AA') &= & (-1)^k \tau (TS') + \tau (ST')\\
& = & (-1)^k (-1)^{k(k'-1)} \tau (S'T) + (-1)^{(k-1)k'} \tau (T'S)\\
& = & (-1)^{kk'} \left[ \tau (S'T) + (-1)^{k'} \tau (T'S)  \right]\\
& = & (-1)^{kk'} \tau_\sigma (A' A).
\end{eqnarray*}
In the same way we have
$$
\tau_\sigma ([\nabla, T] - \sigma [\nabla, S]) =  - \tau ([\nabla, S]) =  - d_B \tau(S) = - d_B \tau_\sigma (T+\sigma S).
$$
\end{proof}

\subsection{The heat semigroup and Duhamel}
The operator $\nabla$ is used to associate with the generalized Dirac operator $D$, different superconnections \cite{Bismut}. They will be for us given by
$$
\A_\sigma (A) := \B_\sigma + A,
$$
where $A$ is a zero-th order fiberwise pseudodifferential operator with coefficients in differential forms of positive degree (in applications $\geq 2$) on $B$. 
Here, the superconnection $\B_\sigma$ is defined by 
$$
\B_\sigma = \nabla + \sigma D.
$$
We are mainly interested in the superconnection $\B_\sigma$ or in the Bismut superconnection together with its metric rescalings. In order to simplify the exposition of our results, we restrict ourselves to the case $\A(0) = \B_\sigma$ and shall briefly explain later how the proofs extend to the general case. 

We have $\B_\sigma^2 = D^2 + X$ where $X= \nabla^2 - \sigma [\nabla, D]$. Note that the operator $X$ is a fiberwise differential operator of order one with coefficients in differential forms of positive degree $\leq 2$. 

\begin{definition}
Following \cite{BGV} we will use the notation $e^{-u \B_\sigma^2}$ to denote the semigroup   (that is, the solution to the heat equation)  given by the following  finite perturbative sum of strong integrals
$$
e^{-u \B_\sigma^2} = \sum_{m\geq 0} (-u)^m \int_{\Delta (m)} e^{-u v_0 D^2} X e^{-u v_1 D^2} \cdots X e^{-u v_m D^2} dv_1 \cdots dv_m.
$$
where $\Delta(m) =\{(u_0, \cdots, u_m) \in \R^{m+1},  \sum u_j =1\}$ is the $m$-simplex.
\end{definition}

Since the base manifold $B$ is finite dimensional, the above sum is finite. Note also that classical results show that the operator $D$ is a self-adjoint regular operator on the Hilbert $C(B, \Lambda^*B)$-module $\maE$, see for instance \cite{Vassout} or \cite{BenameurPiazza}. Hence the heat operator  $e^{-t D^2}$ can be viewed as an adjointable (bounded) operator on $\maE$,
{so that it belongs to the algebra $\maL(\maE)$ of all such operators.} Moreover, since the operator $e^{-t D^2}$ is a smoothing operator, the heat operator $e^{-u \B_\sigma^2}$ associated with the superconnection $\B_\sigma$  defined above, is also a smoothing operator but with coefficients in differential forms of the base $B$. As a consequence, for any $u>0$, the fiberwise graded trace $\tau_\sigma (e^{-u \B_\sigma^2})$ makes sense as a differential form on the base manifold $B$.

\begin{lemma}\label{Duhamel} (Duhamel principle) 
 For any element $A$ of the algebra $\psi^\infty (M|B,E;\Lambda^*B)[\sigma]$,  the  following equality holds in $\maL (\maE)$
$$
[A, e^{-\B_\sigma^2}] = -\int_0^1 e^{- s \B_\sigma^2} [A, \B_\sigma^2] e^{-(1-s) \B_\sigma^2} ds. 
$$ 
\end{lemma}

\begin{proof}
This lemma can be proved following \cite{BGV}. We sketch the argument.
Following \cite{ReedSimon} p. 263-264 the  Duhamel formula is known to be satisfied by the family $D^2$ using just the fact that for $\xi_0$ in our Hilbert space of $L^2$ sections
$\xi_t=e^{-tD^2}\xi_0$ solves the heat equation. If $\xi_0\in \maE$ then using the Schwartz kernel for $e^{-tD^2}$ we see that $\xi_t\in \maE$
from which Duhamel follows in $\maE$.
Now
if we replace $e^{- s \B_\sigma^2}$ and $e^{-(1-s) \B_\sigma^2}$ by the finite expansion sums, we obtain
\begin{multline*}
 e^{- s \B_\sigma^2} [A, \B_\sigma^2] e^{-(1-s) \B_\sigma^2} =  \sum_{m,m'\geq 0} (-1)^{m+m'} s^m (1-s)^{m'} \int_{\Delta (m)\times \Delta(m') } \\ e^{- v_0 s D^2} X e^{- v_1 s D^2} \cdots X e^{- v_m s D^2} [A, D^2 + X] e^{- w_0 (1-s) D^2} X e^{-w_1 (1-s) D^2} \cdots X e^{-w_{m'} (1-s) D^2} dv_1 \cdots dv_m dw_1\cdots dw_{m'}.
\end{multline*}
Now if we integrate over $(0,1)$, make a suitable change of variables using
$$
\sum_{j=0}^m s v_j + \sum_{i=0}^{m'} (1-s) w_i = 1,
$$
and apply the Duhamel principle  for $D^2$ we obtain the result.
\end{proof}

It is worth pointing out that the operator $[A, e^{-\B_\sigma^2}]$ is a fiberwise smoothing operator with coefficients in $\Omega^*(B)$. Moreover, for any $s\in (0,1)$ the operator $e^{- s \B_\sigma^2} [A, \B_\sigma^2] e^{-(1-s) \B_\sigma^2}$ is also fiberwise smoothing. It is then straightforward to check, using the dominated convergence theorem,  that the following holds
$$
\tau_\sigma ([A, e^{-\B_\sigma^2}]) = -\int_0^1 \tau_\sigma (e^{- s \B_\sigma^2} [A, \B_\sigma^2] e^{-(1-s) \B_\sigma^2}) ds. 
$$

\begin{lemma}\label{DuhamelDer}
Consider for any $u\in [0,1]$, a superconnection $\A_u$, given by $\A_u:= \B_\sigma + u A$, where $A$ is a (odd for the grading) zero-th order fiberwise pseudodifferential operator with coefficients in positive degree differential forms on $B$. Then we have
$$
\frac{d}{du} e^{-\A_u^2} = -\int_0^1 e^{- s \A_u^2} [A, \A_u] e^{-(1-s) \A_u^2} ds
$$
 in the strong  operator  topology of $\maL(\maE)$. 
\end{lemma}

\begin{proof}
We set for any small real number $h\not =0$, $Y_u(h):=[ A, \A_u] + h A^2$. Then from the definition of $e^{-\A_v^2}$ we deduce that 
$$
\frac{1}{h} \left[e^{-\A_{u+h}^2} - e^{-\A_{u}^2} \right] = \sum_{k\geq 1} h^{k-1} \int_{\Delta(k)} e^{-u_0\A_{u}^2} Y_u(h) e^{-u_1\A_{u}^2} \cdots Y_u(h) e^{-u_k\A_{u}^2} du_1\cdots du_k.
$$
Both sides are well defined as operators on $\maE$ as they are fiberwise smoothing operators. Applying both sides to elements of $\maE$ we end the proof by letting $h\to 0$.
\end{proof}

We shall also need the following lemma.
\begin{lemma}\label{CommD}
 For any  $A_0, \cdots, A_n$ in the algebra $\psi^\infty (M|B,E;\Lambda^*B)[\sigma]$ we have,
$$
\tau_\sigma \left([\sigma D, A_0 e^{- u_0 \B_\sigma^2} A_1 e^{-u_1 \B_\sigma^2} \cdots A_n  e^{-u_n \B_\sigma^2}]\right) = 0 \text{ for } u_j>0.
$$
\end{lemma}

\begin{proof}
By definition $e^{-u_j \B_\sigma^2}$ is given by the perturbative sum
$$
e^{-u_j \B_\sigma^2} = \sum_{m\geq 0} (-u_j)^m \int_{\Delta (m)} e^{-u_j v_0 D^2} X e^{-u_j v_1 D^2} \cdots X e^{-u_j v_m D^2} dv_1 \cdots dv_m.
$$
For $u_0>0$, we know that the operator $\sigma D A_0 e^{-u_0\B_\sigma/2}$ is fiberwise smoothing with coefficients in differential forms and that its degree is $|A_0|+1$. Therefore, the graded tracial property of $\tau_\sigma$ shows that
\begin{multline*}
 \tau_\sigma \left [  (D A_0 e^{-u_0\B_\sigma^2/2}) (e^{-u_0\B_\sigma^2/2} A_1 e^{-u_1\B_\sigma^2}\cdots A_n e^{-u_n\B_\sigma^2})\right] = \\ (-1)^{(|A_0|+1)\sum_{j=1}^n |A_j|} \tau_\sigma \left [  (e^{-u_0\B_\sigma^2/2} A_1 e^{-u_1\B_\sigma^2}\cdots A_n e^{-u_n\B_\sigma^2}\sigma D)( A_0 e^{-u_0\B_\sigma^2/2})\right]
\end{multline*}
Now as before, the operator $A_0 e^{-u_0\B_\sigma^2/2}$, as well as $$e^{-u_0\B_\sigma^2/2} A_1 e^{-u_1\B_\sigma^2}\cdots A_n e^{-u_n\B_\sigma^2}\sigma D$$ are fiberwise smoothing operators with coefficients in differential forms. Therefore,
\begin{multline*}
 \tau_\sigma \left [  (e^{-u_0\B_\sigma^2/2} A_1 e^{-u_1\B_\sigma^2}\cdots A_n e^{-u_n\B_\sigma^2}\sigma D)( A_0 e^{-u_0\B_\sigma^2/2})\right] = \\(-1)^{|A_0|(1+\sum_{j=1}^n |A_j|)} \tau_\sigma \left [ ( A_0 e^{-u_0\B_\sigma^2/2}) (e^{-u_0\B_\sigma^2/2} A_1 e^{-u_1\B_\sigma^2}\cdots A_n e^{-u_n\B_\sigma^2}\sigma D)\right].
\end{multline*}
Hence we have
$$
\tau_\sigma \left [\sigma D A_0 e^{-u_0\B_\sigma^2} A_1 e^{-u_1\B_\sigma^2}\cdots A_n e^{-u_n\B_\sigma^2}\right] = (-1)^{\sum_{j=0}^n |A_j|} \tau_\sigma \left [A_0 e^{-u_0\B_\sigma^2} A_1 e^{-u_1\B_\sigma^2}\cdots A_n e^{-u_n\B_\sigma^2}\sigma D \right].
$$
\end{proof}

\section{Multilinear functionals and identities}\label{Multi}

In this Section we record some useful identities satisfied by the multilinear
functionals that enter into the bivariant JLO cocycle. {{We shall extensively use ideas developed in the seminal paper \cite{GetzlerSzenes} by E. Getzler and A. Szenes.}}
Let   $A_i\in \psi^\infty (M|B,E;\Lambda^*B)[\sigma]$ and with $\Delta(n)$ being as before, the $n$-simplex, 
we define multilinear functionals \cite{GetzlerSzenes, Wu97} 
$$
\left<\left<A_0, \cdots, A_n\right>\right>_{\B_\sigma} := \int_{\Delta(n)} \tau_\sigma (A_0 e^{- u_0 \B_\sigma^2} A_1 e^{-u_1 \B_\sigma^2} \cdots A_n  e^{-u_n \B_\sigma^2}) du_1 \cdots du_n \quad \in \Omega^*(B).
$$
and 
$$
< A_0, \cdots, A_n> := \int_{\Delta(n)} \tau_\sigma (A_0 e^{-u_0D^2} \cdots A_n e^{-u_nD^2} ) du_1 \cdots du_n \quad \in \Omega^*(B).
$$
The following lemma is stated in the case of flat connections in \cite{Wu97} and is a straightforward extension of \cite{GetzlerSzenes}[Lemma 2.2]. We give the proof for completeness and because it will be used in the sequel.

\begin{lemma}\label{relations}  Let $A_0, \cdots, A_n \in  \psi^\infty (M|B,E;\Lambda^*B)[\sigma]$ and let $\ep_i=(|A_0| + \cdots + |A_{i-1}|)(|A_i| + \cdots + |A_n|)$.
\begin{itemize}
 \item For $1\leq i \leq n$, \quad 
$
\left<\left<A_0, \cdots, A_n\right>\right> = (-1)^{\ep_i} \left<\left<A_i, \cdots, A_n, A_0, \cdots, A_{i-1}\right>\right>;
$
\item 
$
\sum_{i=0}^n \left<\left<A_0, \cdots, A_i, 1, A_{i+1}, \cdots, A_n \right>\right> = \left<\left<A_0, \cdots, A_n\right>\right>;
$
\item 
$
 \left<\left<[\B_\sigma, A_0], A_1, \cdots, A_n\right>\right>   +   \sum_{i=1}^n (-1)^{|A_0|+\cdots+|A_{i-1}|} \left<\left<A_0, \cdots, [\B_\sigma, A_i], \cdots ,A_n\right>\right>
+ d_B \left<\left<A_0, \cdots, A_n\right>\right> = 0;
$
\item For $0\leq i < n$,
$
       \left<\left<A_0, \cdots, A_{i-1} A_i , A_n\right>\right> - \left<\left<A_0, \cdots, A_{i} A_{i+1} , A_n\right>\right>
= \left<\left<A_0, \cdots, [\B_\sigma^2, A_i] , A_n\right>\right>;
      $
 and for $i=n$, {{
$$
\left<\left<A_0, \cdots, A_{n-1} A_n \right>\right> - (-1)^{(|A_0|+\cdots +|A_{n-1}|)|A_n|} \left<\left<A_n A_0,A_1, \cdots, A_{n-1} \right>\right>\\
= \left<\left<A_0, \cdots, A_{n-1},  [\B_\sigma^2, A_n] \right>\right>.
$$
}}
\item Let $(\B_{\sigma, s} = \B_\sigma + sA)_s$  be a $1$-parameter family of superconnections associated with $\sigma D$ as in Lemma \ref{DuhamelDer}, then we have
$$
\frac{d}{ds} \left<\left<A_0, \cdots, A_n \right>\right>_{\B_{\sigma, s}} + \sum_{i=0}^n \left<\left<A_0, \cdots, A_{i}, [\B_{\sigma, s}, A], A_{i+1}, \cdots,  A_n \right>\right>_{\B_{\sigma, s}} = 0.
$$
\end{itemize}
\end{lemma}

\begin{proof}
The first relation is a consequence of the fact that $\tau_\sigma$ is a graded trace. Indeed,  $\B_\sigma^2$ is homogeneous {{ of even total degree}}. The second relation is also clear, see \cite{GetzlerSzenes} for more details. Let us check the third relation. From the perturbative finite sum which defines $e^{-u\B_\sigma^2}$ we have seen in Lemma \ref{Duhamel} that the Duhamel principle holds, hence we obtain the following Bianchi identity
$$
[\B_\sigma, e^{-u\B_\sigma^2}] = -\int_0^1 e^{-u s \B_\sigma^2} [\B_\sigma, \B_\sigma^2] e^{-u(1-s) \B_\sigma^2} ds = 0.
$$ 
Therefore,  the left hand side of the third relation coincides with
$$
\int_{\Delta(n)} \tau_\sigma \left([\B_\sigma, A_0 e^{- u_0 \B_\sigma^2} A_1 e^{-u_1 \B_\sigma^2} \cdots A_n  e^{-u_n \B_\sigma^2}]  \right) du_1 \cdots du_n + d_B \left<\left<A_0, \cdots, A_n\right>\right>.
$$
On the other hand, by Lemma \ref{CommD}
$$
\tau_\sigma ([\sigma D , A_0 e^{- u_0 \B_\sigma^2} A_1 e^{-u_1 \B_\sigma^2} \cdots A_n  e^{-u_n \B_\sigma^2}]) = 0.
$$
Hence, the left hand side of the third relation coincides with
$$
\int_{\Delta(n)} \left[ \tau_\sigma \left([\nabla, A_0 e^{- u_0 \B_\sigma^2} A_1 e^{-u_1 \B_\sigma^2} \cdots A_n  e^{-u_n \B_\sigma^2}]  \right) + d_B \left( \tau_\sigma (A_0 e^{- u_0 \B_\sigma^2} A_1 e^{-u_1 \B_\sigma^2} \cdots A_n  e^{-u_n \B_\sigma^2})\right) \right] du_1 \cdots du_n.
$$
Lemma \ref{CommNabla} then completes the proof of the third item.

The fourth relation is again a consequence of Lemma \ref{Duhamel} and is a straightforward generalization of the similar relation proved for a single operator in \cite{GetzlerSzenes}.

Now, notice that
$$
\frac{d}{ds} \left<\left<A_0, \cdots, A_n \right>\right>_{\B_{\sigma, s}} = \sum_{i=0}^n \int_{\Delta(n)} \tau_\sigma ( A_0 e^{-u_0\B_{\sigma, s}^2} \cdots A_i \frac{d e^{-u_i\B_{\sigma, s}^2}}{ds} A_{i+1} e^{-u_{i+1}\B_{\sigma, s}^2} \cdots A_{n} e^{-u_{n}\B_{\sigma, s}^2}) du_1 \cdots du_n. 
$$
But, Duhamel's formula \ref{DuhamelDer}  shows that
$$
\frac{d e^{-\B^2_{\sigma, s}}}{ds} + \int_0^1 e^{-u \B^2_{\sigma, s}} [\B_{\sigma, s}, A] e^{-(1-u) \B^2_{\sigma, s}} du = 0.
$$
The proof is thus complete.
\end{proof}

\begin{definition}
The bivariant JLO cochain is defined by the sequence $(\psi_n)_n$ given for $(f_0, \cdots, f_n) \in C^\infty(M)^{n+1}$ by the formula
$$
 \psi_n (f_0, \cdots, f_n) := \left< \left< f_0, [\B_\sigma , f_1],  \cdots ,[\B_\sigma , f_n] \right> \right>_{\B_\sigma}.
$$
\end{definition}
One deduces from  Lemma \ref{normS}  that for any $C^1$ function $f$ on the closed manifold $M$, the commutator $[\B_\sigma , f]$ is a bounded operator with values in horizontal $1$-forms.  For simplicity, we  work with smooth functions and smooth forms, although the constructions work obviously with less regularity, and leave it to the interested reader to transpose the statements for more restrictive regularity conditions. 

We recall that $\Omega A$ denotes the universal differential graded algebra of a topological algebra
$A$, that is,
$\Omega^n(A)=A^+\otimes A^{\otimes^n}$ for $n\geq 1$ (where $A^+$ means there is a unit adjoined to $A$). Also recall the operators $b, B$ on $\Omega A$ defined
by:
$$b(a_0da_1\cdots da_n)=\sum_{j=0}^{n-1}(-1)^ja_0da_1\cdots d(a_ja_{j+1})
da_{j+2}\cdots da_n + (-1)^na_na_0da_1\cdots da_{n-1},$$
$$B(\langle a_0\rangle da_1da_2\cdots da_n)=
\sum_{j=0}^n (-1)^{nj} da_j\cdots da_n d\langle a_0\rangle da_1\cdots da_{j-1}.$$ 
respectively. Here  $\langle a_0\rangle $ means either $a_0$ or $1$, where ${1}$ denotes the additional unit. {{Notice that the letter $B$ is already used for the base manifold $B$ but this should not cause any confusion.}}

\begin{lemma}\label{AlgCocycle}
The sequence $\psi=(\psi_n)_{n\geq 0}$ is a chain map of odd degree. More precisely, it is  a chain map from the universal differential algebra of $C^{\infty} (M)$ to the graded Grassmann algebra $\Omega^* (B)$. Indeed, its components whose form degree have the parity of $n$ are trivial, and it satisfies the cocycle relation 
$$
\psi \circ (b+B) + d_B \circ \psi = 0.
$$
\end{lemma}

\begin{remark}
The above lemma is stated under the assumption that the fibers are odd dimensional. When these fibers are even dimensional an analogous result holds.
\end{remark}

\begin{proof}
The third relation in Lemma \ref{relations} applied to $A_0=f_0$ and $A_i=[\B_\sigma, f_i]$ for $i\geq 1$ gives
\begin{multline*}
\left<\left<[\B_\sigma, f_0], \cdots, [\B_\sigma, f_n]\right>\right> - \\ \sum_{i=1}^n (-1)^i \left<\left<f_0, [\B_\sigma, f_1], \cdots, [\B_\sigma^2, f_i], \cdots , [\B_\sigma, f_n]\right>\right>  + \\
d_B \left<\left<f_0, [\B_\sigma, f_1], \cdots, [\B_\sigma, f_n] \right>\right> = 0.
\end{multline*}
On the other hand computing $(B\psi_{n+1}) (f_0, \cdots, f_n)$, we find, using the second item of Lemma \ref{relations}:
$$
(B\psi_{n+1}) (f_0, \cdots, f_n) = \left<\left<[\B_\sigma, f_0], \cdots, [\B_\sigma, f_n]\right>\right>.
$$
Using the last item of Lemma \ref{relations}, we finally deduce that
$$
(b\psi_{n-1})(f_0, \cdots, f_n)  + \sum_{i=1}^n (-1)^i \left<\left<f_0, [\B_\sigma, f_1], \cdots, [\B_\sigma^2, f_i], \cdots , [\B_\sigma, f_n]\right>\right> = 0.
$$
Therefore, we obtain the desired result
$
B\psi_{n+1} + b\psi_{n-1} + d_B \psi_n = 0.
$
\end{proof}
%
\section{The bivariant JLO cocycle is entire}

We are using cyclic homology for bornological algebras due to R. Meyer \cite{MeyerThesis} and will need some preliminaries which we now describe.

\subsection{\label{borno}Review of bivariant entire cyclic homology}

For the convenience of the reader, we summarise in this subsection what we need about bivariant entire homology. The reader is encouraged to consult \cite{MeyerThesis} for more details, especially for the definitions and properties of bornologies. See also \cite{Bourbaki} for the basic (non-trivial) concepts of bornological functional analysis.
Our task here is to adapt this formalism to the families situation.

The idea of using a bornology in the study of entire cyclic cohomology is due to
Connes \cite{ConnesBook} (see pages 370-371). Given a locally convex topological algebra ${A}$, it is proposed there to use the bounded subsets on 
${ A}$ to define entire cyclic cohomology. Meyer develops this idea using bornological functional analysis in a form that is appropriate for this paper in \cite{MeyerThesis}.

Here  $A$, $A_1$, $A_2$ etc will denote complete locally convex topological algebras.  We will use a family of different  bornologies on $A_j$ denoted
generically by
 ${\mathfrak S}(A_j)$. An algebra $A$ equipped with a particular bornology will
be denoted $(A, {\mathfrak S}(A))$.
Denote by $\Omega A$ the universal differential graded algebra of $A$.
Following \cite{MeyerThesis}, Section 3, we introduce the following notions.
\begin{definition}
(i) For $S\in {\mathfrak S}(A)$ $(dS)^\infty$ denotes the union over $n$ of elements
$ds_1 ds_2\ldots  ds_n \in \Omega^{n}(A)$ where $s_1,\ldots ,s_n$ are from
$S$.
Let,  as before, $\langle S\rangle = S\cup \{1\}$ where ${1}$ denotes an additional unit and not the identity of $A$ (we use $A^+$ to denote the adjunction of this unit to $A$)
and then define
$$\langle S\rangle(dS)^\infty= S (dS)^\infty \cup (dS)^\infty \cup S\subset \Omega A,$$
$$S(dS)^{ev} = \langle S\rangle(dS)^\infty\cap \Omega^{ev}A,$$
$$\langle S\rangle(dS)^{odd}= \langle S\rangle (dS)^\infty\cap \Omega^{odd}A.$$
(ii) The notation $\langle a_0\rangle da_1\ldots da_n\in \Omega^nA$ means either $a_0da_1\ldots da_n$
or $da_1\ldots da_n$ depending on context.\\
(iii) ${\mathfrak S}_{an} $ is the bornology on $\Omega A$ generated by 
$\langle S\rangle (dS)^\infty$ for all $S\in {\mathfrak S}(A)$ and $\Omega_{an} A $ denotes the completion of 
$\Omega A$ in the bornology ${\mathfrak S}_{an}$.
Equivalently  ${\mathfrak S}_{an} $ is generated by the union
over $n$ 
of the sets 
$$\{\langle s_0\rangle ds_1\ ds_2  \dots ds_n\vert s_j\in S,
S\in {\mathfrak S}(A)\}.$$
\end{definition}
\begin{remark}
  If $A$ is Fr\'echet, then $A$ is already complete in the bornology given by taking the bounded sets in the Fr\'echet topology, see \cite{Perrot} for instance.
\end{remark}

 In this paper ${ A}$ will always be one of the Fr\'echet algebras  $C^\infty(M)$ or $C^\infty(B)$
where $F\to M\to B$ is a fibration of compact smooth manifolds.
However they will be equipped with bornologies defined by the subsets bounded in 
 certain families of norms.

 If $(V, {\mathfrak S}(V))$ and $(W, {\mathfrak S}(W))$ are complete locally convex bornological spaces then
bounded  linear maps $\ell:V\to W$ are linear maps with the property
that $\ell(S)\in {\mathfrak S}(W)$ whenever $S\in {\mathfrak S}(V)$.
This notion extends to multilinear maps as well.
Moreover
bounded linear maps
$\ell: \Omega_{an}{A}\to W$ 
are in bijection with bounded linear maps on $\Omega{ A}$ equipped with 
the bornology ${\mathfrak S}_{an}$. These in turn are in bijection with
linear maps $\ell:\Omega{ A}\to W$ satisfying 
$\ell(\langle S\rangle(dS)^\infty)\in {\mathfrak S}(W)$, for any $S\in {\mathfrak S}(V)$.

We now explain some key results. We denote by $n!{\mathfrak S}_{an}$
the bornology on $\Omega A$ 
generated by the union over $n$ and $S\in  {\mathfrak S}(A)$ of the sets  
$ n!\langle S\rangle\langle dS\rangle (dS)^{2n}$
which are defined to be
$$\{n!\langle s_0\rangle ds_1\ ds_2  \dots ds_{2n},\vert s_j\in S
\}\cup
\{n!\langle s_0\rangle ds_1\ ds_2  \dots ds_{2n+1}\vert s_j\in S
\}.$$
Let $C({A})$ be the algebra $\Omega A$ completed in the bornology
$n!{\mathfrak S}_{an}$. 
If we equip $\Omega(A)$ with
 the Hochschild boundary $b$ and then with Connes' operator $B$ satisfying the usual relations 
 $b^2=0=B^2=Bb+bB$ then we define a bicomplex. The pair $(b,B)$ extend  to bounded maps on $C(A)$ and $(C(A), b+B)$ is a ${\mathbb Z}_2$-graded complex of complete bornological vector spaces called Connes' entire complex. 
 
 {{An important fact is that   Meyer shows that his analytic cyclic cohomology of $A$
   is the same as Connes' entire cyclic cohomology of $A$. }}
 The idea of the proof is to consider the dual complex $ C(A)'$ of 
bounded linear maps $C(A) \to \mathbb C$. These are just bounded  linear maps $(\Omega A,n!{\mathfrak S}_{an})\to\mathbb C$.  The bounded linear functionals on $(\Omega A,n!{\mathfrak S}_{an})$ are those linear 
maps $\Omega A\to \mathbb C$ that remain bounded on all sets of the form $n!\langle S\rangle\langle dS\rangle (dS)^{2n}$. Identifying 
$\Omega A\cong \sum_{n=0}^\infty\Omega^nA$ and
$\Omega^nA \cong A^+\hat{\otimes}A^{\hat{\otimes}^n}$,
$ C(A)'$ becomes the space of families
$(\phi_n)_{n\in{\mathbb Z}_+}$ of $n+1$-linear maps 
$\phi_n:A^+\times A^n\to\mathbb C$ satisfying the
entire growth condition
$$
\vert \phi_n(\langle a_0\rangle,a_1,\ldots,a_n)\vert
\leq const(S )/[n/2]!
$$  
for all $\langle a_0\rangle \in
\langle S\rangle,a_1  \ldots ,a_n\in S$
and for all $S\in {\mathfrak S}(A)$. Here $[n/2] := k$
 if $n = 2k$ or $n = 2k + 1$ and $const(S )$ is a constant depending 
on $S$ but not on $n$. The boundary on $C(A)$ is composition with $B+b$. 
This motivates us to use the bornological approach of Meyer in the context of Connes $(b,B)$ bicomplex.

 The point of view of \cite{MeyerThesis}
is  to define the bivariant cyclic cohomology of a pair $A_1,A_2$ to be the homology of the complex of bounded linear maps from $\Omega_{an}(A_1)$ to $\Omega_{an}(A_2)$.
In this paper  we replace the analytic bornology {{of $C^\infty (M)$}} by the equivalent entire bornology but 
do not work with the universal graded algebra $\Omega(C^\infty(B))$ {{and}} instead consider the smooth exterior algebra $\Omega^*(B)$ (that is smooth sections of the exterior bundle) associated with the smooth manifold $B$.
 We will equip this smooth algebra with various bornologies which we give in the next subsection. The reason for doing this is that we wish to work with superconnections.

\subsection{{Statement of the main theorem}}

{{The techniques used here are inspired by \cite{BH-JDG}}}. We denote by $C^\ell (M)$, for any $\ell\geq 0$,  the algebra of complex valued functions on the smooth manifold $M$ which are of class $C^\ell$. 
{{The algebra $C^\ell (M)$ can be endowed with a Banach space topology as usual. This is achieved for instance by using local coordinates and a partition 
of unity subordinated with a (finite) open cover.  Using local orthonormal frames extended to vector fields over $M$ using this partition, we can define this 
topology using a finite set ${\mathcal X}$ of vector fields over $M$.}} More precisely, the semi-norms 
$$
p_{q} (f):= {{\sup_{X_j\in {\mathcal X}} }}\|X_1\circ \cdots \circ X_q(f)\|_\infty, \quad 0\leq q \leq \ell,
$$
induce a Banach space topology on $C^\ell (M)$.  {{For simplicity, we have omitted the finite set  ${\mathcal X}$ from the notation. }} 

We shall denote by $\Sigma_\ell$ the bornology on $C^\ell(M)$, {\bf and also its restriction to $C^\infty (M)$}, which is given by the bounded sets of 
the norm $\max_{0\leq q \leq \ell} p_q$. 
{{We also introduce for any vector field $Y$ on $B$,  the notation $d_{Y}$ to denote the operator $i_{Y}\circ d_B$.   
The bornological algebra $(\Omega^*(B), \Sigma_\ell)$ of {{smooth}} differential forms on $B$ is endowed similarly 
with the bornology given by the bounded sets of the usual $C^\ell$ topology on forms.  Recall that this latter is associated with the semi-norms obtained on 
$\Omega^k(B)$ by using  as for $C^\ell (M)$ a finite set ${\mathcal Y}$ of vector fields on $B$ and by considering semi-norms }}
$$
p_r (\omega) := {{\sup_{Y_j \in {\mathcal Y}, \|Z_j\|\leq 1} \frac{1}{2^{kr}}\|  (d_{Y_1} \circ \cdots  d_{Y_r}(i_{Z_1}\circ \cdots \circ i_{Z_k} \omega))\|_\infty, \quad 0\leq r \leq \ell.}}
$$
Our goal  is to prove that the bivariant JLO cochain constructed in the formal spirit of Quillen's seminal paper \cite{Quillen88}, is a 
bounded cyclic cocycle from the entire completion of the universal differential algebra associated with the underlying bornological algebra 
$(C^\infty (M), \Sigma_{\ell+1})$ on the one hand and the bornological algebra $(\Omega^*(B), \Sigma_\ell)$ of {{ smooth}} differential forms on $B$ 
endowed with the $\Sigma_\ell$ bornology on the other hand. Again, we  only consider smooth forms and the restriction of $\Sigma_{\ell}$ to them. As a corollary we shall obtain an entire bivariant cyclic cocycle, following Connes \cite{ConnesBook}.  
To shorten the statements of our results we need some further notation.
Write $n!{\mathfrak S}^{\ell+1}_{an}$ for the entire bornology on $\Omega C^\infty(M)$
 arising from the $\Sigma_{\ell+1}$ bornology on $C^\infty(M)$.
{{
\begin{definition}\label{entiremorph}
\begin{itemize}
\item A morphism $\varphi$ from $\Omega C^\infty(M)$ to $\Omega^*(B)$
is an $\ell$-entire bivariant cochain if it is bounded when $\Omega C^\infty(M)$  is endowed with the entire bornology
$n!{\mathfrak S}^{\ell+1}_{an}$ and $\Omega^*(B)$ with the bornology $ \Sigma_\ell$.
An $\ell$-entire bivariant cocycle is an $\ell$-entire bivariant cochain which satisfies
$\varphi\circ(b+B)+d_B\circ\varphi =0$. 
\item A morphism $\varphi$ from $\Omega C^\infty(M)$ to $\Omega^*(B)$
is called an entire bivariant cochain here if it is bounded when $\Omega C^\infty(M)$  is endowed with the entire bornology
$n!{\mathfrak S}^{\infty}_{an}$ and $\Omega^*(B)$ with the bornology $ \Sigma_\infty$.
\end{itemize}
\end{definition}
}}

{{\begin{remark}
A morphism $\varphi$  is a bivariant entire cochain if and only if it is a bivariant $\ell$-entire cochain, for all $\ell\geq 0$.
\end{remark}
}}

We are now ready to state our first theorem. Recall that the fibers of our fibration are odd dimensional. There is a similar statement in the even case.
\begin{theorem}\label{JLO} For any integer $\ell\geq 0$, the bivariant JLO cochain
$\psi$ is an $\ell$-entire bivariant cocycle. 
\end{theorem}

\begin{corollary}
The bivariant JLO cochain  is entire with respect to the first variable for the Fr\'echet $C^\infty$-topology of $C^\infty(M)$ and $C^\infty(B)$.
\end{corollary}

\begin{proof}
If $S$ is a bounded set for the Fr\'echet $C^\infty$ topology of $C^\infty(M)$, then for any $\ell\geq 0$, $S$ is bounded in the $C^{\ell+1}$ topology. Therefore, if $A$ is a subset of $\langle S\rangle(dS)^\infty$ then applying Theorem \ref{JLO}, its image under the morphism defined by $\psi$ will be contained in some set $\langle S'\rangle(dS')^\infty$ for a bounded set $S'$ in the $C^\ell$ topology of $\Omega^*(B)$. Since this is true for any $\ell \geq 0$, we deduce that $\psi (A)$ is bounded for the $C^\infty$ topology. 
\end{proof}

\begin{corollary}\label{pairing}
For any $U\in GL_N(C^\infty (M))$ and for any $\ell \geq 0$, the following series of differential forms on $B$ converges in the $C^\ell$-topology to a closed differential form whose cohomology class is denoted $<\JLO(D), U>$:
$$
\sum_{k\geq 0} (-1)^k k!  \left<\left< U^{-1}, [\B_\sigma, U] , \cdots, [\B_\sigma, U^{-1}], [\B_\sigma, U]   \right>\right>_{2k+1}.
$$
\end{corollary}

\begin{proof}
This corollary is the precise rephrasing of the following fact. Since the bivariant JLO cochain is entire and closed, it  pairs with Connes' \cite{ConnesJLO} entire cyclic cycles of $C^\infty (M)$ to yield a closed differential form on $B$, and direct inspection of the pairing of \cite{ConnesJLO} gives precisely the one in the statement of the corollary.
\end{proof}

Recall that a sequence $(\phi_n)_{n\geq 0}$ of cochains $\phi_n: C^\infty (M)^{\otimes_{n+1}}\to \C$ is entire in the sense of Connes' definition \cite{ConnesBook} for the $C^s$ norm $\|\cdot\|_{s}$ if and only if  for any bounded set $S$ in $(C^\infty (M), \|\cdot\|_s)$ there exists a constant $C(S)$ such that
$$
|\phi_n(f_, \cdots, f_n)| \leq C(S)/[n/2]!, \quad \text{ for any } f_i\in S \text{ and any }n\geq 0.
$$
This allows us to define in the same way entire cocycles for the Fr\'echet topology.
Another consequence of Theorem \ref{JLO}  is the following:

\begin{corollary}
Let $C$ be a closed de Rham current on the base manifold $B$ of degree $N\in \{0, \cdots, \dim (B)\}$. Then the following sequence 
$$
\psi^C =(\int_C \psi_n)_{n-N \in 2\Z+1},
$$
is an entire cyclic cocycle on the algebra $C^\infty (M)$. Moreover, the following series converges in $\C$ to the pairing of the Chern character of $U$ with the composition of $\JLO(D)$ with $C$:
$$
\sum_{k\geq 0} (-1)^k k!  \int_C \left<\left< U^{-1}, [\B_\sigma, U] , \cdots, [\B_\sigma, U^{-1}], [\B_\sigma,  U]   \right>\right>_{2k+1} = \int_C \left< JLO (D), U \right>.
$$
\end{corollary}

\begin{proof}\ 
Computing $(b+B)\psi^C$ we find
$$
b\psi^C_{n-1} + B\psi^C_{n+1}= <C, b\psi_{n-1} + B \psi_{n+1}> = - <C, d_B \psi_n> = 0.
$$
The last equality is true since $C$ is closed. It remains to check the entire property. Notice that the closed current $C$ defines a cyclic cocycle on $C^\infty (B)$ which yields a $C^\infty$ -continuous graded trace on $\Omega^*(B)$. Composing this trace with the JLO bivariant cocycle, we conclude using Corollary \ref{pairing}. 
\end{proof}

\section{Proof of Theorem \ref{JLO}}

The proof is long and will be split into many subparts. 
\subsection{Estimates}

The proof of Theorem \ref{JLO} rests on establishing some estimates on our bivariant JLO functional. We collect the
preliminary facts in this subsection.
We denote by $d_vf$ the differential of $f$ in the fiberwise direction and by $d_Hf$ the differential of $f$ in the horizontal direction defined by the horizontal distribution $H$. We choose the metric on $M$ so that $H$ and the fiberwise bundle $T_vM$ are orthogonal. Recall that if $c$ is the fiberwise Clifford representation then for $f\in C^\infty (M)$ we have $[D, f] = c(d_vf)$. Note also that
$[\nabla, f] = d_Hf\wedge \cdot$

For a vector field $Y$ on $B$, we denote by ${\tilde Y}$ the horizontal vector field on $M$ satisfying $\pi_*{\tilde Y} = Y$. 
Recall that $\pa$ denotes the graded (with respect to the degree of the forms) commutator associated with the quasi-connection $\nabla$, a (exterior) graded derivation of the algebra $\psi^\infty (M|B, E; \Lambda^*B)$ of fiberwise pseudodifferential operators with coefficients in horizontal differential forms. We denote, for any horizontal (i.e. $H$ valued) vector field $Z$ on $M$, by $\nabla_Z$ the composition $i_Z\circ \nabla$ where $i_Z$ is contraction by $Z$.  As usual, for $P\in \psi^h (M|B, E; \Lambda^kB)$ we also denote by $\pa_{Z} (P)$ the element $[\nabla_Z, P]$ of $\psi^h (M|B, E; \Lambda^kB)$. 

If $P\in \psi^h (M|B, E)$ with $h\leq 0$, then we set 
$$
\|P\| := \sup_{b\in B} \|P_b\|,
$$ 
where the norm $\|P_b\|$ is the operator norm on the $L^2$ sections. In general, if $ P\in \psi^h (M|B, E; \Lambda^kB)$ for $h\leq 0$ and $k\geq 0$ we define the uniform norm of $P$ by the same expression, except that now $\|P_b\|$ is obtained by taking the supremum over  $k$-multivectors $Z\in \Lambda^k(T_bB)$ of norm $\leq 1$, of the operator norms $\|i_Z P_b\|$.
We use the operators $(1+D_b^2)^{s/2}$ to define the Sobolev pre-Hilbert $H^s$ topology on
$C^\infty(M_b, E\vert_{M_b})$. We may extend the definition given above for operators of zeroth or negative order to operators of positive order $\alpha$.  The $H^s$-norm of an operator $A_b$ of order $\alpha$ is defined as the norm of $A_b:H^s\to H^{s+\alpha}$ and these norms are comparable for all $s$. By taking the supremum of over $b\in B$ as above we obtain comparable norms
for different choices of $s$
for operators of order $\alpha$. In the discussion below we will for convenience use these $s$-norms interchangeably without comment.

\begin{lemma}\label{alphaq}
For any $q\geq 0$, 
$$
{{\sup_{Y_1, \cdots , Y_r\in  {\mathcal Y}}}} \|(I+D^2)^{-1/2} [\pa_{\tilde{Y}_1} \cdots \pa _{\tilde{Y}_r}] (D^2)(I+D^2)^{-1/2} \| = \alpha_r(D^2) < +\infty,
$$
and 
$$
{{\sup_{Y_1, \cdots , Y_r\in  {\mathcal Y}}}} \|(I+D^2)^{-1/2} [\pa_{\tilde{Y}_1} \cdots \pa _{\tilde{Y}_r}] (D) \| = \alpha_r(D) < +\infty $$
where the $Y_j$'s are vector fields on $B$. 
\end{lemma}

\begin{proof} The same proof works for both operators and we only give the proof for $D^2$. 
We first point out that the operator $[\pa_{Y_1} \cdots \pa _{Y_q}] (D^2)$ is a second order fiberwise differential operator, with smooth coefficients. Therefore, the operator 
$$
(I+D^2)^{-1/2} [\pa_{\tilde{Y}_1} \cdots \pa _{\tilde{Y}_q}] (D^2)(I+D^2)^{-1/2},
$$
is a zero-th order fiberwise pseudodifferential operator whose norm is  finite. Moreover, as  $B$ is compact, by a partition of unity argument we may assume that 
we are given a local orthonormal basis $(\pa_1, \cdots, \pa_b)$ of the tangent bundle to $B$ over an open set $U\subset B$. 
Then, we can replace the operators $\pa_{Y_j}$ by operators of the form $\tilde\pa_j:=\pa_j + \omega(\pa_j)$, where $\omega$ is a matrix of differential  $1$-forms. Now, 
 the finite family of operators $(I+D^2)^{-1/2} [\tilde\pa_{j_1} \cdots \tilde\pa _{j_q}] (D^2)(I+D^2)^{-1/2} $, for $1\leq j_1, \cdots, j_q\leq b$, is uniformly 
bounded over $U$. {{Since the vector fields $Y_1, \cdots Y_q$ belong to the finite family ${\mathcal Y}$}}, the proof is thus complete.
\end{proof}

\begin{lemma}\label{EstimateNabla2}
For any $r\geq 0$, 
$$
\sup_{\|Z_1\|\leq 1, \|Z_2\|\leq 1, Y_1, \cdots , Y_r\in {\mathcal Y}} \|(I+D^2)^{-1/2} [\pa_{\tilde{Y}_1} \cdots \pa _{\tilde{Y}_r}] (i_{\tilde{Z}_1\wedge \tilde{Z}_2} \nabla^2) \| = \beta_r (\nabla, D) < +\infty,
$$
where $Z_1, Z_2, Y_1, \cdots , Y_r$ are vector fields on $B$ that we view through their unique horizontal lifts.
\end{lemma}

\begin{proof}
The proof follows the same lines as the previous lemma. More precisely, using the compactness of $B$ we can reduce to local coordinates. But  as can 
be checked in these local coordinates, since the operator $i_{Z_1\wedge Z_2} \nabla^2$ is a smooth family of differential operators of order $1$, 
the smooth family of zero-th order operators $(I+D^2)^{-1/2} [\pa_{Y_1} \cdots \pa _{Y_q}] (i_{Z_1\wedge Z_2} \nabla^2)$ is uniformly bounded. 
\end{proof}

Recall that for $f\in C^\infty (M)$, $
\|f\|_s := \max_{0\leq j\leq s} p_j(f).$

\begin{lemma}\label{normS}
For any $s \geq 0$, there exists a constant $C_s\geq 0$ such that 
$$
\|[\pa_{\tilde{Y}_s}\circ \cdots \circ \pa_{\tilde{Y}_1}] (f)\| \leq C_s \|f\|_s \text{ and } \|[\pa_{\tilde{Y}_s}\circ \cdots \circ \pa_{\tilde{Y}_1}] ([D, f])\|\leq C_s \|f\|_{s+1},
$$
for any $f\in C^\infty(M)$ and any vector fields {{$Y_j$ from the finite family ${\mathcal Y}$}}.
\end{lemma}

\begin{proof}
For any $j\leq s$, we have:
$$
[\pa_{\tilde{Y}_j}\circ \cdots \circ \pa_{\tilde{Y}_1}] (f) = [\tilde{Y}_j \circ \cdots \circ \tilde{Y}_1] (f).
$$
The RHS means the multiplication operator by the function $[\tilde{Y}_j \circ \cdots \circ \tilde{Y}_1] (f)$. Since $M$ is compact, there obviously exists a constant 
$C_s>0$  only depending on the distribution $H$ such that
$$
\|[\tilde{Y}_j \circ \cdots \circ \tilde{Y}_1] (f)\| \leq C_s p_j(f), \quad \text{ for any } j\leq s.
$$
Using the fact that $[D,f]$ is Clifford multiplication by $df$ we can expand in local coordinates
to easily prove in a similar fashion the second estimate.
\end{proof}

\begin{lemma}\label{p=N}\
Fix any $\ep\in ]0, 1/2]$, then for fiberwise pseudodifferential operators $(A_j)_{0\leq j \leq N}$ and $(B_j)_{0\leq j \leq N}$ with $A_j\in \psi^0(M|B, E)$ and $B_j\in \psi^2(M|B, E)$ for any $j$, we have:
$$
\|<A_0, B_0, \cdots, A_N, B_N>\| \leq \left(\frac{\pi}{2\epsilon}\right)^{N+1}\frac{{\| \tau (e^{-(1-\ep)D^2})\|}}{N!} \times \Pi_{j=0}^N \|A_j\| \|(I+D^2)^{-1/2}B_j (I+D^2)^{-1/2}\|.
$$
\end{lemma}


\begin{proof}
By inspection we have
$$
<A_0, B_0, \cdots, A_N, B_N> = < A_0(I+D^2)^{1/2}, (I+D^2)^{-1/2}B_0 , \cdots , A_N (I+D^2)^{1/2}, (I+D^2)^{-1/2}B_N>.
$$
Therefore, using H\"older's inequality fiberwise, we obtain   
(writing $d\underline{u}=du_1 \cdots du_N, d \underline{v} =dv_0 dv_1\cdots dv_N$):

$\|<A_0, B_0, \cdots, A_N, B_N>\| $ 
\begin{eqnarray*} & \leq & \int_{\Delta (2N+1)} \Pi_{j=0}^N \|A_j (I+D^2)^{1/2} e^{-u_jD^2}\|_{1/u_j}  \|(I+D^2)^{-1/2} B_j e^{-v_jD^2}\|_{1/v_j}  d\underline{u} d\underline{v} \\
& \leq & \Pi_{j=0}^N \|A_j\| \|(I+D^2)^{-1/2} B_j (I+D^2)^{-1/2}\| \\ && \int_{\Delta (2N+1)} \Pi_{j=0}^N \|(I+D^2)^{1/2} e^{-u_jD^2}\|_{1/u_j} \|(I+D^2)^{1/2} e^{-v_jD^2}\|_{1/v_j} d \underline{u}d \underline{v} \\
& \leq & \Pi_{j=0}^N \|A_j\| \|(I+D^2)^{-1/2} B_j (I+D^2)^{-1/2}\| \\ && \int_{\Delta (2N+1)} \Pi_{j=0}^N \|(I+D^2)^{1/2} e^{-u_j\ep D^2}\|  \|(I+D^2)^{1/2} e^{-v_j\ep D^2}\| \|\tau(e^{-(1-\ep)D^2})^{u_j} \tau(e^{-(1-\ep)D^2})^{v_j}\| d \underline{u}d \underline{v}
\end{eqnarray*}
\begin{eqnarray*}
&\leq & \| e^{-(1-\ep)D^2}\|_1 \Pi_{j=0}^N \|A_j\| \|(I+D^2)^{-1/2} B_j (I+D^2)^{-1/2}\| \\ & & \int_{\Delta (2N+1)} \Pi_{j=0}^N \|(I+D^2)^{1/2} e^{-u_j\ep D^2}\|  \|(I+D^2)^{1/2} e^{-v_j\ep D^2}\|  d \underline{u}d \underline{v}.
\end{eqnarray*}
But, for any $\alpha >0$, we have by  the spectral theorem in $\maL(\maE)$  (see for instance \cite{BenameurPiazza}),
$$
\|(I+D^2)^{1/2} e^{-\alpha D^2}\| \leq \frac{e^{\alpha - 1/2}}{\sqrt{2\alpha}}.
$$
Therefore, 
$$
\Pi_{j=0}^N \|(I+D^2)^{1/2} e^{-u_j\ep D^2}\|  \|(I+D^2)^{1/2} e^{-v_j\ep D^2}\| \leq \frac{e^{\ep -1/2} }{(2\ep)^{N+1}} \times \Pi_{j=0}^N (u_jv_j)^{-1/2}.
$$
Now we complete the proof by computing the following integral:
$$
\int_{\Delta (2N+1)} \Pi_{j=0}^N (u_j v_j)^{-1/2} d \underline{u}d \underline{v} = \frac{\pi^{N+1}}{N!}.
$$
\end{proof}

We shall also need the intermediate estimate corresponding to $p+1$ entries of second order fiberwise pseudodifferential operators in $< \cdots >_{n+p+1}$. In fact a similar method of proof establishes our next result.

\begin{lemma}\label{p<N}\
For any $\ep\in ]0, 1/2]$, for any $A_0, \cdots , A_N\in \psi^0(M\vert B, E)$ and any $B_{j_0}, \cdots , B_{j_p}\in \psi^2(M|B, E)$ with $p<N$ and $0\leq j_0 < \cdots < j_p\leq N$, the following estimate holds
\begin{multline*}
\left\| \left< A_0, \cdots , A_{j_0}, B_{j_0}, A_{j_0+1}, \cdots, A_{j_1}, B_{j_1}, \cdots , A_{j_p}, B_{j_p}, A_{j_p+1}, \cdots, A_N\right> \right\| \leq \\ \left(\frac{\pi}{2\ep}\right)^{p+1}\frac{\|\tau (e^{-(1-\ep)D^2})\| }{ N!} \times \Pi_{i=0}^N \|A_i\| \Pi_{i=0}^p \|(I+D^2)^{-1/2} B_{j_i} (I+D^2)^{-1/2} \|.
\end{multline*}
\end{lemma}

\begin{proof}
We apply again the method of proof of Lemma \ref{p=N} and use the equality
$$
\int_{u_0+\cdots + u_N+ v_{j_0}+ \cdots + v_{j_p} = 1} \frac{du_1\cdots du_N dv_{j_0}\cdots dv_{j_p}}{\sqrt{u_{j_0} \cdots u_{j_p} v_{j_0} \cdots v_{j_p}}} = \frac{\pi^{p+1} p!}{N!}.
$$
More precisely, we have
\begin{multline*}
\| \tau (A_0e^{-u_0D^2} \cdots A_{j_0}e^{-u_{j_0}D^2}  B_{j_0} e^{-v_{j_0}D^2} \\ A_{j_0+1}e^{-u_{j_0+1}D^2}  \cdots A_{j_p}e^{-u_{j_p}D^2} B_{j_p}e^{-v_{j_p}D^2}  A_{j_p+1}e^{-u_{j_p+1} D^2}  \cdots A_Ne^{-u_{N}D^2})\|  \leq \\ \Pi_{i=0}^N \|A_i\| \Pi_{i=0}^p \|(I+D^2)^{-1/2} B_{j_i}  (I+D^2)^{-1/2}\| \|e^{-(1-\ep)D^2}\|_1  \Pi_{i=0}^p \|(I+D^2)^{1/2} e^{-u_{j_i}\epsilon D^2}\| \| (I+D^2)^{1/2} e^{-v_{j_i}\epsilon D^2}\|. 
\end{multline*}
Next we apply the spectral theorem in $\maE$ to estimate
$$
\| (I+D^2)^{1/2} e^{-u_{j_i}D^2}\| \| (I+D^2)^{1/2} e^{-v_{j_i}D^2}\| \leq \frac{e^{-1/2 + \ep u_{j_i}}}{\sqrt{2u_{j_i}\ep}} \frac{e^{-1/2 + \ep v_{j_i}}}{\sqrt{2v_{j_i}\ep}} \leq \frac{1}{\sqrt{2u_{j_i}\ep}\sqrt{2v_{j_i}\ep}}.
$$
The rest of the proof is straightforward.
\end{proof}

\subsection{{Last steps of the proof of the theorem}}
Recall from \cite{MeyerThesis} that the universal differential graded algebra $\Omega C^{\infty}(M)$ is endowed with the entire bornology $\Sigma_{\ell+1}$ generated by the sets $n!\langle S\rangle (dS)^\infty$ where $S$ describes the bounded subsets of $C^{\infty}(M)$ for the 
$C^{\ell+1}$  topology recalled in the beginning of subsection \ref{borno}. Recall  that $\Omega^*(B)$ is similarly endowed with the  bornology given by the bounded sets for the $C^\ell$ topology on smooth forms.

In order to estimate the semi-norms of $\psi_N (f_0, \cdots, f_N)$, we need to expand into its homogeneous components. We denote by $\maJ$ the subset of $\{0,1\}^3$ given by
$$
\maJ = \{(1,0,0); (0,1,0); (0,0,1) \}
$$
For $\alpha\in \maJ$ we denote by $\alpha^{(j)}$ the $j$-th component of $\alpha$, $j=1,2,3$. So only one of the integers $\alpha^{(j)}$ is non trivial and equals $1$. We shall set  $b^{\alpha^{(j)}}$ in a given expression to mean that  when $\alpha^{(j)} =1$, we take into account $b$ but when $\alpha^{(j)} =0$ then we simply erase $b$ from the expression. For instance
$$
(a_0, \cdots, a_k, b^{\alpha^{(j)}}, a_{k+1}, \cdots, a_n),
$$
equals the $(n+2)$-tuple $(a_0, \cdots, a_k, b, a_{k+1}, \cdots, a_n)$ when $\alpha^{(j)} =1$ and the $(n+1)$-tuple $(a_0, \cdots,  a_n)$ when $\alpha^{(j)} =0$. For $\alpha \in \maJ$, we set
$$
X^{\alpha} (b) := [\nabla, b]^{\alpha^{(1)}} (\sigma [\nabla, D])^{\alpha^{(2)}} \nabla^{2\alpha^{(3)}}.
$$
For any  $m\geq 0$, $n=(n_0, \cdots n_m)\in \N^{m+1}$ and $\alpha=(\alpha_1, \cdots, \alpha_m)\in \maJ^m$, we define an  $\sum_{j=0}^m n_j + \sum_{i=1}^m \alpha_i^{(1)}$ cochain $\phi^m_{\alpha, n}$ with values in $ m + \sum_{i=1}^m \alpha_i^{(3)}$ differential forms on $B$, by the formula
\begin{multline*}
 \phi^m_{\alpha, n} (f_0, \cdots, f_{n_0}, g_1^{\alpha_1^{(1)}}, f_{n_0+1}, \cdots, f_{n_0+n_1}, \cdots, g_m^{\alpha_m^{(1)}}, f_{n_0+n_1 + \cdots +n_{m-1} + 1}, \cdots , f_{n_0 + \cdots + n_m} ) := \\ < f_0, \sigma [D,f_1], \cdots , \sigma [D,f_{n_0}], X^{\alpha_1}(g_1), \sigma [D, f_{n_0+1}], \cdots , \sigma [D, f_{n_0+n_1}], \\
 \cdots , X^{\alpha_m}(g_m) , \sigma [D, f_{n_0+\cdots n_{m-1}+1}], \cdots , \sigma [D,f_{n_0+\cdots n_m}] >. 
\end{multline*}

\begin{lemma}
The cochains $\psi_{N}$ of the JLO cocycle can be expanded as a finite algebraic sum over  $m\geq 0$, $n=(n_0, \cdots n_m)\in \N^{m+1}$ and $\alpha=(\alpha_1, \cdots, \alpha_m)\in \maJ^m$ of the bihomogeneous cochains $\phi^m_{\alpha, n}$. Moreover the number of
such $\phi^m_{\alpha, n}$ is {{ bounded by $(\dim B)^{N+1}2^{\dim B}$}}.
\end{lemma}

\begin{proof}
We first replace in $\psi_N$, each factor $e^{-u_j\B_\sigma^2}$ by its definition, a finite perturbative sum,  and by  using a straightforward change of variables, we easily deduce that $\psi_N (f_0, \cdots, f_N)$ is a finite signed sum over $(m_0, \cdots, m_N)\in \N^{N+1}$ of the terms 
$$
\left< f_0; \stackrel{m_0\text{ times}}{\overbrace{X, \cdots, X}}; [\B_\sigma, f_1]; \stackrel{m_1\text{ times}}{\overbrace{X, \cdots, X}}; \cdots ; [\B_\sigma, f_N]; \stackrel{m_N\text{ times}}{\overbrace{X,\cdots , X}}\right>
$$
Now,  replacing $X$ by its value $\nabla^2 - \sigma [\nabla, D]$ and $[\B_\sigma, f_j]$ by its value $[\nabla, f_j]- \sigma [D, f_j]$, it is easy to rewrite each such term as a finite signed sum of appropriate  $\phi^m_{\alpha, n}$'s.

Next we see that $\sum_{j=o}^N m_j$ is bounded by the dimension $\dim B$ of the base
because we cannot have more than $\dim B$ differential forms in any term.
If we set 
$$
\lambda_N=\sharp \{(m_0,m_1,\ldots, m_N)\ \vert \ \sum_j m_j \leq \dim B\}
$$
and $|m|= m_0+\ldots +m_N$
then necessarily we have  $2^{|m|} \times \lambda_N\leq (\dim B)^{N+1}2^{\dim B}$ as required.
\end{proof}

We have chosen to  expand the JLO cocycle as a finite combination of bihomogeneous cochains with respect to the cochain grading and the form grading. By doing so, our formulae are explicit enough to be paired with closed currents on the base.

{{
\begin{proposition}
Set $N= n_0+\cdots +n_m + m$, then for $N$ large, the bihomogeneous cochain $\phi^m_{\alpha, n}$ can be estimated as follows:
$$
p_r \left(\phi^m_{\alpha, n} (f_0, \cdots, f_{N})\right) \leq 
\frac{C'(S)^{N+1}}{N!}, \quad 0\leq r\leq \ell\text{ and } f_j\in S.
$$
where $C'(S)$ is some constant which only depends on the bounded set $S$ for the $C^{\ell+1}$ topology on $C^\infty (M)$.  
\end{proposition}
}}

\begin{proof}  
For simplicity, we denote by  $i_Y$ either contraction by the vector field $Y$ over $B$, or by  its horizontal lift on $M$ or  on $\psi^\infty (M|B,E;\Lambda^*B)$ and $\psi^\infty (M|B,E;\Lambda^*B)[\sigma]$. Then for fiberwise smoothing operators $T\in \psi^{-\infty} (M|B,E;\Lambda^*B)[\sigma]$, we have
$$
(i_Y \circ \tau_\sigma)(T) = -(\tau_\sigma\circ i_Y)(T) \text{ and } (d_B\circ \tau_\sigma) (T) + (\tau_\sigma \circ \pa) (T)=0.
$$
Therefore, if $d_Y:= i_Y\circ d_B$ is the derivative in the direction $Y$ in $B$, then
$$
(d_Y\circ \tau_\sigma) (T) - (\tau_\sigma\circ \pa_Y)(T) = 0$$
We need to estimate for $0\leq s \leq \ell$, the semi-norms 
$$
p_s (\phi^m_{\alpha, n} (f_0, \cdots, f_N )) \text{ where } N= n_0+\cdots +n_m + m \text{ odd},
$$
for given $f_0, \cdots , f_N$ in a bounded set $S$ for the $C^{\ell+1}$ topology. So, we assume that there exists a constant $C\geq 0$ such that $\|f_j\|_t\leq C$ for any $0\leq t\leq \ell+1$ and for $0\leq j \leq N$. \\

For the convenience of the reader, we first explain the proof for $m=0$ thus giving a guide to the general case.

\underline{Step I: $m=0$}\\
We begin by estimating the $p_s$ semi-norms of functions on $B$. They are given by
$$
\phi_N^0 (f_0, \cdots, f_N) = \left< f_0, \sigma [D, f_1], \cdots, \sigma [D, f_N]\right >.
$$
Let  $\|\cdot \|_\alpha$ denote the  supremum over $B$ of the fiberwise $\alpha$-Schatten norm of a compact  operator on $L^2$ sections. Using H\"older's inequality for each $b\in B$ and taking the supremum over $B$, we have:
\begin{eqnarray*}
\| \phi_N^0 (f_0, \cdots, f_N)\| & \leq & \int_{\Delta(N)}  \|\tau \left( f_0 e^{-u_0D^2} [D, f_1] e^{-u_1D^2}\cdots [D, f_N] e^{-u_ND^2} \right)\|  du_1\cdots du_N \\ &\leq &  \int_{\Delta(N)} \|f_0\| \|e^{-u_0D^2}\|_{1/u_0} \|[D, f_1]\| \|e^{-u_1D^2}\|_{1/u_1} \cdots \|[D, f_N]\| \|e^{-u_ND^2}\|_{1/u_N}\\ & \leq & \frac{\|e^{-D^2}\|_1}{N!} \times \|f_0\| \Pi_{j=1}^N \|[D_b, f_j]\|.
\end{eqnarray*}

Using Lemma \ref{normS}, we deduce
\begin{eqnarray*}
\| \phi_N^0 (f_0, \cdots, f_N)\| 
& \leq  &  C_0^N \frac{\|e^{-D^2}\|_1}{N!}\times \|f_0\| \|f_1\|_1 \cdots \|f_N\|_1\\
& \leq &  \|e^{-D^2}\|_1\frac{C_0^N C^{N+1}}{N!}.
\end{eqnarray*}\\
In the same way, let $Y_1, \cdots, Y_s$ be vector fields on $B$ {{taken from the finite collection ${\mathcal Y}$}}. Using Lemma \ref{CommNabla}, we can write
$$
d_{Y_1}  \cdots d_{Y_s} \phi_N^0 (f_0, \cdots, f_N)  =  \int_{\Delta(N)} \tau \left( [\pa_{Y_1}  \cdots \pa_{Y_s}] (f_0 e^{-u_0D^2} [D, f_1] e^{-u_1D^2}\cdots [D, f_N] e^{-u_ND^2}) \right) du_1\cdots du_N.
$$
Note that the operators $\pa_{Y_{j_1}}\cdots \pa_{Y_{j_k}} (f)$ and $\pa_{Y_{j_1}}\cdots \pa_{Y_{j_k}} [D,f]$ are zero-th order differential operators on $M$ and by Lemma \ref{normS}, we can estimate
$$
\|\pa_{Y_{j_1}}\cdots \pa_{Y_{j_k}} (f)\|\leq C_k \|f\|_k \text{ and } \|\pa_{Y_{j_1}}\cdots \pa_{Y_{j_k}} [D,f]\| \leq C_k \|f\|_{k+1}.
$$
Therefore, one can apply the H\"older inequality exactly, as in the case $s=0$ treated above, and deduce the required estimates for all the terms that involve no derivatives of the fiberwise smoothing operators $e^{-u_jD^2}$. Thus, we may concentrate on terms of the form
$$
\tau \left( A_0 [\pa_{Y_{k^0_1}}\cdots \pa_{Y_{k^0_{\beta_0}}}] e^{-u_0D^2} A_1 [\pa_{Y_{k^1_1}}\cdots \pa_{Y_{k^1_{\beta_1}}}] e^{-u_1D^2}\cdots A_N [\pa_{Y_{k^N_1}}\cdots \pa_{Y_{k^N_{\beta_N}}}] e^{-u_ND^2}] \right),
$$
where $A_j$ is a zero-th order pseudodifferential operator, $0\leq \beta_l \leq s$ and $\Sigma \beta_l$ is at most $s$ and is prescribed by the number of derivatives applied to get the operators $A_j$ out of the operators $f_0$ and $[D, f_j]$. Now, apply again  Duhamel's formula:
$$
\pa_Y e^{-uD^2} = -u \int_0^1 e^{-utD^2} \pa_Y (D^2) e^{-u(1-t)D^2} dt = -\int_0^u e^{-tD^2} \pa_Y (D^2) e^{-(u-t)D^2} dt.
$$
For instance,
\begin{multline*}
\int_{\Delta(N)} \tau (A_0 \pa_Y e^{-u_0D^2} A_1 e^{-u_1D^2} \cdots A_Ne^{-u_ND^2}) du_1\cdots du_N \\ =  \int_{\Delta(N+1)} \tau (A_0 e^{-v_0D^2} \pa_Y (D^2) e^{-v_1D^2} A_1 e^{-v_2D^2} \cdots A_Ne^{-v_{N+1}D^2}) dv_1\cdots dv_{N+1}
\\ =  \left< A_0, \pa_Y (D^2), \sigma A_1, \cdots , \sigma A_N\right>_{N+1}
\end{multline*}
So, the norm of this term can be estimated using Lemma \ref{p<N} and Lemma \ref{alphaq}. More pecisely, we get for any $\ep\in ]0,1/2]$ (one can take here $\ep=1/2$ for simplicity)
\begin{multline*}
\| \left< A_0, \pa_Y (D^2), A_1, \cdots , A_N\right>_{N+1} \| \\ \leq \frac{\pi \|e^{-(1-\ep)D^2}\|_1} {\ep N!} \|(I+D^2)^{-1/2} \pa_Y(D^2) (I+D^2)^{-1/2}\| \Pi_{i=0}^N \|A_i\|\\ \leq \frac{\pi   \alpha_1(D^2) \|\tau(e^{-(1-\ep)D^2})\|} {\ep N!} \Pi_{i=0}^N \|A_i\|
\end{multline*}
The other, more complicated, terms are dealt with in a similar way using Lemma \ref{p<N} and the method of Step III where we show how to establish the bound:
$$
\| \tau_\sigma \left( A_0 [\pa_{Y_{k^0_1}}\cdots \pa_{Y_{k^0_{\beta_0}}}] e^{-u_0D^2} A_1 [\pa_{Y_{k^1_1}}\cdots \pa_{Y_{k^1_{\beta_1}}}] e^{-u_1D^2}\cdots A_N [\pa_{Y_{k^N_1}}\cdots \pa_{Y_{k^N_{\beta_N}}}] e^{-u_ND^2}] \right)\|  \leq \frac{C (\ep) }{N!} \Pi_{i=0}^N \|A_i\|.
$$
Again, the operators  $A_j$ are here derivatives of $f_0$ or $[D,f_j]$'s. \\

\underline{Step II: general $m$}\\
The general case involves differential forms on the base manifold $B$ obtained from commutators of functions with $\nabla$, commutators of $D$ with $\nabla$ and also with  the curvature $\nabla^2$. The latter 
presents some difficulties. The commutators of functions with $\nabla$ are easy to handle, as they give zero-th order differential operators and can be estimated using the H\"older inequality again and Lemma \ref{normS}. On the other hand, commutators of $D$ with $\nabla$ and terms involving $\nabla^2$ introduce additional derivatives in the fiberwise direction, as these operators are first order fiberwise differential operators with coefficients in differential forms of degree $1$ for the first and degree $2$ for the second. So, we cannot apply directly the argument of Lemma \ref{p<N} and we need to give careful estimates for such terms. The worst situation arises when the entries $B_0, \cdots , B_N$ in the expression $<B_0, \cdots , B_N>$ are composed of `too many' fiberwise pseudodifferential operators of positive orders. By this we mean that, in addition to the $B_j$'s, we have the maximum number of commutators $[\nabla, D]$ or $[\nabla^2,D]$ and also the maximum number of directional derivatives of the heat kernel $e^{-u_jD^2}$. The latter introduce operators of order $2$.  Fortunately, and this seems to be a crucial point here,  the base manifold is finite dimensional and we are taking at most $\ell$ directional derivatives. Hence the number of entries involving $[\nabla, D]$ or $[\nabla^2, D]$ is limited by the dimension and the number of derivatives is also bounded by $\ell$.

Denote by $k$ the degree of the differential form 
$$
\phi_{\alpha, n}^m (f_0, \cdots, f_{n_0}, g_1^{\alpha_1^{(1)}}, \cdots, g_m^{\alpha_m^{(1)}}, \cdots, f_{\sum_{i=0}^m n_i}).
$$
So, $k= m+ \sum_{i=1}^m \alpha_i^{(3)}$. We fix vector fields $Z_1, \cdots, Z_k$ on the base manifold $B$ with norms $\leq 1$ and thus need to estimate, for $s\leq \ell$, the $s$ seminorm of the function $i_{Z_1}\cdots i_{ Z_k} \phi_{\alpha, n}^m ( \cdots)$. This reduces to the computation of the $s$ seminorm of terms $<A_0, \cdots,  A_{N+q+q'}>$ 
where $N$ entries $A_j$ are zero-th order, $q$ entries are first order and $q'$ entries are second order, and where as explained above, we can assume that $N$ is as large as allowed, while $q$ and $q'$ are bounded by $\sup(\ell, \dim B)$. Indeed, for $N\geq q+q'$, we obtain the desired estimate as follows. 

First to illustrate the ideas assume that the order is as follows:
$$
< A_0, B_1, A_1, \cdots, B_q, A_q, C_{1}, A_{q+1}, \cdots, C_{q'}, A_{q+q'}, A_{q+q'+1}, \cdots, A_N>
$$
where the $A_j$' are zero-th order, the $B_j$'s are first order and the $C_j$' are second order. By using the H\"older inequality, we reduce to the
issue of  estimating the  expression 
$$
\| e^{-uD^2} E e^{-vD^2} A\|_{1/(u+v)} \leq \|e^{-u D^2} E (I+D^2)^{-1/2}\|_{1/u} \|(I+D^2)^{1/2} e^{-vD^2} A\|_{1/v}.
$$
where $E$ is at most second order. This gives for any $\ep\in ]0,1/2]$
$$
\| e^{-uD^2} E e^{-vD^2} A\|_{1/(u+v)}$$
$$ \leq   \|e^{-u \ep D^2} (I+D^2)^{1/2}\| \|(I+D^2)^{-1/2} E (I+D^2)^{-1/2}\|  \times \|(I+D^2)^{1/2} e^{-v\ep D^2} A\| \sup_b \tau(e^{-(1-\ep)D_b^2})^{u+v}.
$$
Now, again we have by the spectral theorem
$$
\|e^{-u \ep D^2} (I+D^2)^{1/2}\| \leq \frac{e^{\ep -1/2}}{\sqrt{2u\ep}},
$$
and the estimate goes exactly as for the previous simpler cases. We thus obtain the existence of a constant $C(\ep)\geq 0$ such that 
\begin{multline*}
\|< A_0, B_1, A_1, \cdots, B_q, A_q, C_{1}, A_{q+1}, \cdots, C_{q'}, A_{q+q'}, A_{q+q'+1}, \cdots, A_N>\| \\ \leq  \frac{C(\ep)}{N!} \Pi_{i=0}^N \|A_i\| \Pi_{i=1}^q \|(I+D^2)^{-1/2} B_i\| \Pi_{i=1}^{q'} \|(I+D^2)^{-1/2} C_i (I+D^2)^{-1/2}\|.
\end{multline*}

\underline{Step III: the worst case}\\
To indicate how to handle general terms we now consider
$$
< A_0, A_1, \cdots, A_k, A_{k+1}, \cdots, A_{r}, A_{r+1}, \cdots, A_N>
$$
where $A_1, \cdots, A_k$ are of order two, the next $r-k$  are of order one
and the last $N-r+1$ are of order zero. We may assume in this expression that $N$ is chosen so that
$k+r \leq (N-1)/2=N^\prime$.
The integrand of the JLO type functional in this instance may be written as
\begin{multline*}
\prod_{j=0}^k [(I+D^2)^{j/2}A_j(I+D^2)^{-j/2-1}(I+D^2)^{1/2}e^{-u_jD^2}]\\
\times\prod_{j=k+1}^r[(I+D^2)^{(k+1)/2}A_j (I+D^2)^{-(k+1)/2 -1/2}(I+D^2)^{1/2} e^{-u_jD^2}]\\
\ \ \ \ \ \ \ \ \ \ \ \ \ \  \times \prod_{i=0}^{r+k} (I+D^2)^{(k+1-i)/2} A_{r+i+1}(I+D^2)^{-(k+1-i)/2}(I+D^2)^{1/2}e^{-u_{r+i+i}D^2}\\
\times A_{r+2+k}e^{-u_{r+k+2}D^2}\cdots A_Ne^{-u_ND^2}
\end{multline*}
{{Now we may integrate over the simplex and estimate the norm of the resulting expression and we find that it is bounded by
\begin{multline*}
\prod_{j=0}^N \|A_j(I+D^2)^{-\alpha_j/2}\|
\|\tau(e^{-(1-\epsilon)D^2}\|\int_{\Delta(N)}\prod_{j=0}^{r+k+1}(2\epsilon u_j)^{-1/2} du_1\cdots du_N\\
= \|\tau(e^{-(1-\epsilon)D^2}\|(2\epsilon)^{-(r+k)/2+1}\prod_{j=0}^N
\|A_j(I+D^2)^{-\alpha_j/2}\|\int_{\Delta(N)}(u_0\cdots u_{r+k+1})^{-1/2}du_1\cdots du_N
\\= \|\tau(e^{-(1-\epsilon)D^2}\|(2\epsilon)^{-(r+k)/2+1}\prod_{j=0}^N
\|A_j(I+D^2)^{-\alpha_j/2}\|\frac{\pi^{(r+k+2)/2}}{\Gamma(\frac{r+k}{2}+1)(N-(r+k))!}\beta((r+k)/2+1, N-(r+k+1))
\end{multline*}
where we have used
$$\int_{\Delta(N)}(u_0\cdots u_{\ell})^{-1/2}du_1\cdots du_N=
\frac{\pi^{(\ell+1)/2}}{\Gamma(\frac{\ell}{2}+1)\Gamma(N-\ell)}\beta((\ell+1)/2, N-\ell).$$
We can assume, to simplify the evaluation of the beta function w.l.o.g., that $r+k$
is even say $2\gamma$.
Then, using the expression for the beta function in terms of gamma functions, the previous expression is bounded by
$$
\left(\frac{\pi}{2\epsilon}\right)^{\gamma+1}\|\tau(e^{-(1-\epsilon)D^2})\|\times\prod_{j=0}^N\|A_j(1+D^2)^{-\alpha_j/2}\|\frac{\gamma!}{(N-(\alpha +1))!}
$$
Now $\gamma\leq N'$, and $N-(\gamma +1)-\gamma\geq N'$
so that we can estimate the ratios of gamma functions and bound the preceding expression by
$$
\left(\frac{\pi}{2\epsilon}\right)^{\gamma+1}\frac{\|\tau(e^{-(1-\epsilon)D^2})\|} {(N-1) ! (N-\gamma)(N-\gamma-1)}\times\prod_{j=0}^N
\|A_j(1+D^2)^{-\alpha_j/2}\|
$$
This estimate obviously suffices to  deduce the allowed estimate.}}
\end{proof}

{{Now, in order to deduce the proof of Theorem \ref{JLO}, we point out that $\psi_N$ is a sum of at most $(\dim B)^{N+1}\times 2^{\dim B}$ components $\phi_{\alpha, n}^m$. Therefore, using the Stirling estimate, it is easy to deduce the existence of a constant $C(S)$ depending only on the bounded set $S$ of $(C^\infty(M), \Sigma_{\ell+1})$ such that
$$
p_r \left(\psi_N (f_0, \cdots, f_N)\right) \leq \frac{C(S)}{[N/2]!}, \quad 0\leq r\leq \ell.
$$
}}

\subsection{More general superconnections and transgression}

We show in this subsection how to extend Theorem \ref{JLO} to more general superconnections associated with the odd operator $\sigma D$ and prove that the entire bivariant cyclic homology class does not depend on certain choices made in the course of the argument. Since the techniques are classical, we shall be brief. More precisely, we consider superconnections $\A$ given as
$$
\A := \B_\sigma + A\text{ where } A\text{ is an odd element of } \Psi^0(M|B, E; \Lambda^*B)[\sigma],
$$
whose {{differential form degrees are positive}}. Recall that $\B_\sigma = \sigma D + \nabla$ so that $\A= \sigma D + \nabla + A$. 
Given such a superconnection $\A$, we can write $\A^2= {D}^2 + X'$ where $X'=\nabla^2 + A^2 +  [ \nabla, \sigma D + A]$ has only positive degree forms and is therefore nilpotent.

We define the heat kernel $e^{-\A^2}$ of the superconnection $\A$ by the usual finite Duhamel expansion where we simply replace the operator $X$ by $X'$:
$$
e^{-\A^2} := \sum_{m\geq 0}  \int_{\Delta (m)} e^{- v_0 {D}^2} X' e^{- v_1 {D}^2} \cdots X' e^{- v_m {D}^2} dv_1 \cdots dv_m.
$$
where $\Delta(m) =\{(u_0, \cdots, u_m) \in \R^{m+1},  \sum u_j =1\}$ is again the $m$-simplex. 
As with the superconnection $\B_\sigma$ which corresponds to $A=0$, we define for any odd integer $n$:
$$
\psi_n (f_0, \cdots, f_n) := \left<\left< f_0, [\A, f_1], \cdots, [\A, f_n]  \right> \right>_\A
$$
\begin{proposition}
Given a superconnection $\A$ associated with  $\sigma D$ as above, the cochains $(\psi_n)_n$ form, for all $\ell\geq 0$, an $\ell$-entire bivariant cochain $\JLO(\A)$ from the universal graded algebra of $C^\infty (M)$ to the graded algebra $\Omega^*(B)$ of differential forms on $B$,  in the sense of Definition \ref {entiremorph} and hence $\JLO(\A)$ is also entire.
\end{proposition}

\begin{proof}
Notice first that the algebraic relations proved in  Lemma \ref{relations} (all of them besides  the last one) are still valid with $\A$ replacing $\B_\sigma$. Hence, the collection $(\psi_n)_n$ is again a bivariant cyclic cocycle by exactly the same proof as for Lemma \ref{AlgCocycle}. The proof of boundedness is a rephrasing of the proof of \ref{JLO}. The only difference is that we have to deal with new terms involving $A^2+ [\B_\sigma, A]$.  The term $A^2 + [\nabla, A]$ is a zero-th order fiberwise pseudodifferential operator with coefficients in positive degree forms and causes no trouble.  We only have to explain how to estimate terms involving $-\sigma [D, A]= -\sigma \sum_{k>0} [D, A_{[k]}]$, where $A_{[k]}$ is the component of $A$ which increases the form degree by $k$. But this is done using the following modification of Lemma \ref{EstimateNabla2} and which is proved in the same way by reducing to local coordinates:
$$
{{\sup_{\|Z_1\|\leq 1, \cdots , \|Z_k\|\leq 1, Y_1, \cdots , Y_q\in {\mathcal Y}}}} \|(I+D^2)^{-1/2} [\pa_{Y_1} \cdots \pa _{Y_q}] (i_{Z_1\wedge\cdots  Z_k} [D, A_{[k]}]) \| = \beta'_q < +\infty,
$$
We omit the proof here. Then the rest of the proof is tedious but is exactly a rephrasing of the proof given in the previous subsection. 
\end{proof}

\begin{remark}
We show in Theorem \ref{cohomologous} that the $\ell$-entire cohomology class of $\JLO(\A)$ coincides with the $\ell$-entire cohomology class of $\JLO (\B_{\sigma})$.
\end{remark}

We now proceed to prove the main result of this subsection, namely the transgression formula for our JLO entire bivariant cocycle. We follow the method adopted in \cite{GetzlerSzenes}. 
Set, for any superconnection $\A$ associated with $\sigma D$ as above, and with $V$  a homogeneous fiberwise pseudodifferential operator  with coefficients in differential forms on the base $B$:
$$
\Ch (\A, V) (f_0, \cdots, f_n) := \sum_{i=0}^n (-1)^{i|V|} \left<\left< f_0, [\A, f_1], \cdots, [\A, f_i], V, [\A, f_{i+1}], \cdots , [\A, f_n]\right>\right>_{\A},
$$
and
$$
\alpha^* (\A, V) (f_0, \cdots , f_n) := \sum_{i=1}^n (-1)^{(i-1)(|V|+1)} \left<\left< f_0, [\A, f_1], \cdots, [V, f_i], \cdots, [\A, f_n]\right>\right>_{\A}.
$$

\begin{lemma}
\begin{itemize}
 \item Assume that the pseudodifferential order of $V$ is $\leq 1$, then $\Ch (\A, V)$ is an $\ell$-entire bivariant cochain, for any $\ell\geq 0$.
\item Assume that the pseudodifferential order of $V$ is $\leq 0$, then $\alpha^* (\A, V)$ is
an $\ell$-entire bivariant cochain, for any $\ell\geq 0$.
\end{itemize}
\end{lemma}

\begin{proof}
{{ We only give the proof for $\A=\B_\sigma$ and leave the general case as an exercise. }}
Let us now prove for instance the first item, the second being easier since $V$ is bounded. Applying the definition of $e^{-u\B_\sigma^2}$  in the expression 
$$
\left<\left< f_0, [\B_\sigma, f_1], \cdots, [\B_\sigma, f_i], V, [\B_\sigma, f_{i+1}], \cdots , [\B_\sigma, f_n]\right>\right>_{\B_\sigma},
$$
we reduce to the estimates of terms of the form{{
\begin{multline*}
< f_0, X, \cdots, X, [\B_\sigma, f_1], X , \cdots, X, \cdots, [\B_\sigma, f_i], X, \cdots, X; V; \\  X, \cdots, X, [\B_\sigma, f_{i+1}], X, \cdots, X, \cdots, [\B_\sigma, f_{n}], X, \cdots, X>
\end{multline*}
}}The point is to apply the argument of  Steps II and III in the proof of Theorem \ref{JLO}, by simply adding one operator of order $1$ in the entries. Recall  that this can be done as long as the number of operators of order $1$ or $2$ is not too big with respect to the number of operators of order $0$. But, notice that $V$ only appears once, the first order operator $X$ has coefficients in differential forms of positive degree only and hence cannot appear more than the dimension of $B$ times. Therefore, since we only need the estimates for $n$ large, the same proof works and we obtain the required estimates exactly as in Steps II and III  of the proof of Theorem \ref{JLO}. Notice that we have to estimate the sum of $n+1$ terms of the form $$
\left<\left< f_0, [\B_\sigma, f_1], \cdots, [\B_\sigma, f_i], V, [\B_\sigma, f_{i+1}], \cdots , [\B_\sigma, f_n]\right>\right>_{\B_\sigma},
$$ 
{{but since the estimate involves $1/n!$, we get  the allowed estimate of the kind $\frac{C(S)}{[n/2]!}$.}}
\end{proof}

\begin{proposition}\label{Identity}
The following identity holds
$$
(d_B + (-1)^{|V|}(b+B)) \Ch(\B_\sigma, V) + \Ch (\B_\sigma, [\B_\sigma, V]) + (-1)^{|V|} \alpha^*(\B_\sigma, V) = 0.
$$
\end{proposition}

\begin{proof}
For $0\leq i \leq n$, we apply the third relation of Lemma \ref{relations} to the operators
$$
A_0=f_0, A_j=[\B_\sigma, f_j]\text{ for }1\leq j \leq i, A_{i+1}=V \text{ and } A_j= [\B_\sigma, f_{j-1}]\text{ for } j\geq i+2.
$$
So, for $i=0$ for instance, this means that we apply that relation to
$
A_0=f_0, A_1= V $ and $A_j= [\B_\sigma, f_{j-1}]$ for $ j\geq 2$.
For any $i$ this gives us a relation $\theta_i + d_B \theta'_i=0$ where $(-1)^{i|V|}\theta_i=X_1^i+X_2^i+X_3^i =0$ where
$$
X_1^i= (-1)^{i|V|} \left<\left< [\B_\sigma, f_0], \cdots,  [\B_\sigma, f_i], V, [\B_\sigma, f_{i+1}], \cdots, [\B_\sigma, f_n]  \right> \right>
$$
\begin{multline*}
 X_2^i= \sum_{1\leq j\leq i} (-1)^{i|V|+j-1} \left<\left<  f_0, [\B_\sigma, f_1], \cdots, [\B_\sigma^2, f_j], \cdots, [\B_\sigma, f_i], V,  [\B_\sigma, f_{i+1}], \cdots,  [\B_\sigma, f_n]\right> \right> + \\
\sum_{j=i+1}^n (-1)^{j-1+(i+1)|V|} \left<\left<  f_0, [\B_\sigma, f_1], \cdots, [\B_\sigma, f_i], V,  [\B_\sigma, f_{i+1}], \cdots, [\B_\sigma^2, f_j], \cdots, [\B_\sigma, f_n]\right> \right>
\end{multline*}
and
$$
X_3^i=(-1)^{i(|V|+1)} \left<\left<  f_0, [\B_\sigma, f_1], \cdots, [\B_\sigma, f_i], [\B_\sigma, V],  [\B_\sigma, f_{i+1}], \cdots,  [\B_\sigma, f_n]\right> \right>.
$$
Finally, $\theta'_i$ is given by
$$
\theta'_i = \left<\left<  f_0, [\B_\sigma, f_1], \cdots, [\B_\sigma, f_i], V, [\B_\sigma, f_{i+1}], \cdots,  [\B_\sigma, f_n]\right> \right>.
$$
Thus the expression $\sum_{i=0}^n (-1)^{i|V|} (\theta_i + d_B \theta'_i) = 0$ allows us to write
$$
\sum_{i=0}^n X_1^i + \sum_{i=0}^n X_2^i + \sum_{i=0}^n X_3^i + d_B \Ch (\B_\sigma, V) =0.
$$
Now we have,  by inspection, the relation
$$
\sum_{i=0}^n X_3^i = \Ch(\B_\sigma, [\B_\sigma, V]) (f_0, \cdots, f_n).
$$
Similarly, we leave it to the reader to directly compute $\sum_{i=0}^n X_2^i$. One finds
{{$$
\sum_{i=0}^n X_2^i = (-1)^{|V|} [b\Ch(\B_\sigma, V) + \alpha^* (\B_\sigma, V)] (f_0, \cdots, f_n).
$$}}
Next, using the second relation of Lemma \ref{relations}, we see that
$$
B \Ch(\B_\sigma, V) (f_0, \cdots, f_n) = \sum_{i=0}^n (-1)^{(i+1)|V|} \left<\left<  [\B_\sigma, f_0], \cdots, [\B_\sigma, f_i], V, [\B_\sigma, f_{i+1}], \cdots,  [\B_\sigma, f_n]\right> \right>.
$$
The conclusion follows immediately.
\end{proof}

We are now in position to prove 

\begin{theorem}\label{cohomologous}
Let $\A$ be, as before, the  superconnection $\A := \B_\sigma + A$. Then the JLO  $\ell$-entire cocycle $\JLO(\A)$ associated with the superconnection $\B_\sigma +A$ is cohomologuous to the JLO $\ell$-entire cocycle $\JLO(\B_\sigma)$ associated with the superconnection $\B_\sigma = \nabla + \sigma D$. 
\end{theorem}

\begin{proof}
Let $\B_{\sigma, s} := \B_\sigma + sA$ be the smooth linear path of superconnections associated with $\sigma D$. Then we can write, using the fifth relation of Lemma \ref{relations}:
\begin{multline*}
 \frac{d}{ds} \left<\left< f_0, [\B_{\sigma,s}, f_1], \cdots ,  [\B_{\sigma,s}, f_n]   \right>\right>_{\B_{\sigma, s}} = \\
-\sum_{i=0}^n  \left<\left< f_0, [\B_{\sigma,s}, f_1], \cdots ,  [\B_{\sigma,s}, f_i], [\B_{\sigma,s}, A],   [\B_{\sigma,s}, f_{i+1}], \cdots, [\B_{\sigma,s}, f_n] \right>\right>_{\B_{\sigma, s}} + \\ 
\sum_{i=1}^n \left<\left< f_0, [\B_{\sigma,s}, f_1], \cdots [\B_{\sigma,s}, f_{i-1}],  [A, f_i],  \cdots, [\B_{\sigma,s}, f_n] \right>\right>_{\B_{\sigma, s}}
\end{multline*}
Notice that $|A|=1$ while $|[\B_{\sigma, s}, A]| = 2$. Hence we obtain
$$
\frac{d}{ds} \left<\left< f_0, [\B_{\sigma,s}, f_1], \cdots ,  [\B_{\sigma,s}, f_n]   \right>\right>_{\B_{\sigma, s}} = 
-\Ch (\B_{\sigma, s}, [\B_{\sigma, s}, A]) (f_0, \cdots, f_n) + \alpha^* (\B_{\sigma,s}, A) (f_0, \cdots, f_n).
$$
But we know from Proposition \ref{Identity} that
$$
-\Ch (\B_{\sigma, s}, [\B_{\sigma, s}, A])  + \alpha^* (\B_{\sigma,s}, A) = [d_B - (b+B)] \Ch(\B_{\sigma, s}, A),
$$
which completes the proof since $\Ch(\B_{\sigma, s}, A)$ is an even $\ell$-entire cochain.
\end{proof}

\section{Compatibility with the higher spectral flow}

\subsection{Higher spectral flow}
An application of our entire JLO cocycle comes from its relation with the higher spectral flow. Using our previous results we explain in this Section  how to prove the equality between the Chern character of the higher spectral flow and the corresponding JLO pairing, which is well defined in the Fr\'echet topologies. 
Higher spectral flow, introduced in \cite{DaiZhang} {{(see also \cite{LP})}} for a family of fiberwise self-adjoint elliptic operators $D=(D_b)_{ b\in B}$, is only well defined under the assumption that the $K^1$ class defined by the family is trivial. We assume this from now on. Our proof that the Chern character of higher spectral flow coincides with the pairing with our entire  bivariant JLO cocycle is a  generalization of Getzler's proof in the case of a single operator \cite{GetzlerOdd}. Recall that the fiberwise generalized Dirac operator $D$ defines a class $[D]$ in the Kasparov group $KK^1(M,B)$ \cite{Kasparov}, and hence using the Kasparov product, a homomophism $K^1(M) \to K^0(B)$ which assigns to $U\in K^1(M)$ the class $U\cap [D]$. This Kasparov product is an index map which is described below using either families of Toeplitz operators or the notion of higher spectral flow. 
Denoting by $E$ the entire cyclic homology we prove in the present Section commutativity of the following diagram:

\vspace{0.1in}
\begin{picture}(415,90)

\put(115,75){$K^1(M)$}
\put(125,65){ $\vector(0,-1){35}$}
\put(100,48){$\Ch $}

\put(185,85){$\SF(D, \cdot)$}
\put(170,79){$\vector(1,0){60}$}

\put(185,25){$\JLO(D)$}
\put(170,19){$\vector(1,0){60}$}

\put(235,75){$K^0(B)$}
\put(255,65){ $\vector(0,-1){35}$}
\put(270,48){$\ch$}

\put(95,15){$HE_1 (C^\infty(M))$}

\put(235,15){$H^{even} (B, \C)$}
\end{picture} 

Thus the $\ell$-entire cyclic cohomology class $\JLO(D)$ is precisely the bivariant Chern-Connes character of $[D]$.
Combining our result with the main result of \cite{DaiZhang}, we deduce that $\JLO(D)$ coincides up to the (obviously bounded) Hochschild-Kostant-Rosenberg-Connes (HKRC) map for $M$, with the topological map $H^{odd}(M) \to H^{even} (B)$ given up to constant by
$$
\omega \longmapsto \int_{M/B} \omega\wedge {\hat A} (TM|B).
$$
By using the results of \cite{MelrosePiazza}, we know that the $K^1$ index of $D=(D_b)_{b\in B}$ is zero if and only if there exists a (smooth) spectral section $P$ for $D$, that is, a smooth family of self-adjoint fiberwise pseudodifferential projections $P=(P_b)_{b\in B}$ acting on the $L^2$-sections such that for some smooth non-negative function $\varrho$ on $B$,
$$
P_b  1_{]\varrho (b), +\infty)} (D_b) = P_b \text{ and } P_b 1_{]-\infty, -\varrho (b)[} (D_b) = 0, \quad \forall b\in B.
$$ 
The following result is  taken from \cite{MelrosePiazza}.
\begin{proposition}\cite{MelrosePiazza}
 Let $D=(D_b)_{b\in B}$ be as before the fiberwise {{generalized }}Dirac operator along the smooth fibration $\pi: M\to B$ and assume that the index of $D$ in $K^1(B)$ is trivial. Given a spectral section $P$ for $D$, there exists a self-adjoint fiberwise zero-th order pseudodifferential  operator $A\in \Psi^{0} (M|B; E)$  such that for any $b\in B$ the operator $D_b+A_b$ is invertible and $P_b$ coincides with $1_{[0, +\infty)} (D_b+A_b)$.
\end{proposition}

We assume from now on that the operator $A$ is chosen as in the previous proposition and thus associated with a fixed spectral section $P$ for $D$. So $P=1_{[0, +\infty) } (D+A)$ and $D+A$ is a zero-th order perturbation of $D$ and is a family of invertible operators. 

\begin{proposition}
For any $U\in GL_N(C^\infty (M))$, the operator 
$$
PUP:= (P\otimes 1_N)\circ U \circ (P\otimes 1_N),
$$
acting fiberwise on the image of $L^2(M_b,E)\otimes \C^N$ under the projection $P_b\otimes 1_N$, is a smooth family of Fredholm operators whose index class in $K^0(B)$ is denoted $\Ind (T_U)$.  Then, $\Ind(T_U)$  does not depend on the choice of the spectral section $P$ and only depends on the $K^1$ class of $U$.  
\end{proposition}

\begin{proof} Compare with \cite{DaiZhang}.
For any fixed $b\in B$, the usual proof for a single Toeplitz operator on the odd dimensional closed manifold $M_b$ shows that 
$$
P_bU_bP_b= (P_b\otimes 1_N)\circ U|_{M_b} \circ (P_b\otimes 1_N)
$$ 
 is a Fredholm operator in the Hilbert space $P_b(L^2(M_b, E|_{M_b}))^N$.  Hence, $PUP$ is a smooth family of Fredholm operators on the image of $P$. We then know that the homotopy class of this family  defines a class in $K^0(B)$, see \cite{AtiyahBook}. To explicitly define this class as the Atiyah-Singer index of a fiberwise elliptic operator, we follow \cite{BaumDouglas} and define the zero-th order fiberwise elliptic pseudodifferential operator $T_{U,P}$ by
$$
T_{U,P} := I-P + PUP.
$$
The index $\Ind(PUP)$ is, by definition, the Atiyah-Singer index class of $T_{U,P}$ in $K^0(B)$ \cite{AtiyahSinger4}. If we choose another spectral section $P'$ then the usual computation shows that the operator $T_{U,P'}$ is a perturbation of 
$T_{U, P}$, therefore the index class is unchanged. A homotopy class of invertibles $U_t$ yields a homotopy class of principal symbols of the fiberwise operator $T_{U,P}$ and hence the index is unchanged.
\end{proof}

\begin{definition}\cite{DaiZhang}
Assume that $[0,1]\ni t\mapsto D_t:=(D_{t,b})_{b\in B}$ is a smooth path of fiberwise elliptic pseudodifferential operators such that the index class of the endpoints, $D_0$ and $D_1$ in {{$K^1(B)$}}, are trivial. Choose spectral sections $P_0, P_1$ for $D_0, D_1$ respectively and fix a spectral section $Q= (Q_t)_{t\in [0,1]}$ for the total family viewed as a fiberwise operator over $B\times [0,1]$. Then the spectral flow of the path $(D_t)_{t\in [0,1]}$ with respect to $P_0$ and $P_1$ is the class in $K^0(B)$ defined by:
$$
 \SF (D; P_0, P_1) := [P_1- Q_1] - [P_0 - Q_0].
$$
\end{definition}

It is easy to check that $ \SF(D; P_0, P_1)$ does not depend on the choice of the global spectral section $Q$. In this paper we are mainly interested in the affine path $D_t:= D + t U^{-1} [D, U]$ where $D$ is a family of generalized Dirac operators over $B$ whose index class in $K^1(B)$ is trivial, and $U$ is a given element of $GL_N(C^\infty (M))$. In this case the endpoints are conjugate and we consider the spectral flow with respect to the spectral sections $P_0=P$ and $P_1= U^{-1} P U$, where $P$ is a fixed spectral section for $D$. It turns out that the spectral flow does not depend on $P$ either and is an invariant of the principal symbol of $D$ and of the homotopy class of $U$. We denote it $ \SF(D,U)$. Indeed, Dai and Zhang proved the following

\begin{proposition}\cite{DaiZhang}.
We have in $K^0(B)$, $
\Ind (T_U) =  - \SF (D, U).$
\end{proposition}

\subsection{Second theorem and reduction to the third theorem}

The pairing of $\ell$-entire bivariant cyclic homology with entire cyclic homology with respect to $\Sigma_\ell$ reads in our case as follows:
$$
 <\JLO (D) , U > := \sum_{n\geq 0} (-1)^n n! \left<\left<U^{-1},[\B_\sigma, U],   \cdots, [\B_\sigma, U^{-1}], [\B_\sigma, U]\right>\right>_{\B_\sigma,2n+1}
$$

\begin{theorem} \label{JLO=SF} Assume that the index class of $D$ in $K^1(B)$ is trivial. Then for any $U\in GL_N(C^\infty (M))$, the following relation holds in the even de Rham cohomology of the base manifold $B$:
$$
\frac{1}{{\sqrt\pi}} <\JLO (D) , U >  = \ch ( \SF (D,U)) = - \ch (\Ind (T_U)).
$$
where $\ch$ is the usual Chern character on the manifold $B$.
\end{theorem}
The last relation is clear from the previous proposition and the fact that the Chern character only depends on the $K$-theory class. The proof of this theorem is long and we split it into several lemmas and propositions.

\begin{lemma} 
Let $\A$ be a  superconnection associated with the operator $\sigma D$ as in the previous sections. Define the affine path of superconnections $(\A_t)_{0\leq t\leq 1}$ given by
$$
\A_t := \A + t U^{-1} [\A, U].
$$
Then \ 
(i) the differential form $\int_0^1 \tau _\sigma ( U^{-1} [\A, U] e^{-\A_t^2}) dt$ is a closed form on $B$,\\
 (ii)  the cohomology class of $\int_0^1 \tau _\sigma ( U^{-1} [\A, U] e^{-\A_t^2}) dt$ does not depend on the choice of superconnection $\A$. 
\end{lemma}

\begin{proof}
(1)  Let $\tilde\A:= dt \frac{\pa}{\pa t}  + \A_t$ be the associated  superconnection for the fibration $M\times [0,1] \to B\times [0,1]$. Therefore, the differential form $
\tau _\sigma (e^{-\A^2})$ is closed in $B\times [0,1]$. A straightforward computation using the fact that $\tau _\sigma (e^{-\A_t^2})$ is itself closed in $B$, proves the following relation
$$
d_B \left(\tau_\sigma(\dot{\A_t} e^{-\A_t^2})\right) = \frac{d}{d t} \tau _\sigma (e^{-\A_t^2}),
$$
Hence,
$$
d_B \int_0^1 \tau _\sigma (\dot{\A_t} e^{-\A_t^2}) dt = \tau _\sigma (e^{-\A_1^2}) - \tau _\sigma (e^{-\A_0^2}) = \tau _\sigma (U^{-1} e^{-\A_0^2} U) - \tau _\sigma (e^{-\A_0^2}) = 0.
$$
The last equality is deduced from the relation $\sigma U = U\sigma$ and the graded tracial property of the functional $\tau$. 

(2) Assume that we are given another superconnection $\B'$ associated with $\sigma D$. Consider the corresponding affine path $\B'_t:= \B' +tU^{-1} [\B', U]$ as before, and  the smooth family $\A_{t,s} = \B_t + s (\B'_t - \B_t)$ of superconnections associated with $\sigma D$, where $s$ also runs over $[0,1]$. We then set
$$
\D := \A_{t,s} + dt \frac{\pa}{\pa t} + ds \frac{\pa}{\pa s}.
$$
Clearly, $\D$ is a superconnection associated with $\sigma D$ but for the smooth fibration $M \times [0,1]^2\to B \times [0,1]^2$. Therefore, the differential form $
\tau _\sigma (e^{-\D^2})$ is closed in $B\times [0,1]^2$. Using this fact and that 
$$
d_B \tau _\sigma (e^{- \A_{t,s}^2}) = 0,
$$
we obtain the relation
$$
d_B \tau_\sigma (e^{- \A_{t,s}^2} \wedge K^2_{t,s}) =  dt\wedge ds \left[ \frac{\pa}{\pa s} \tau _\sigma (\frac{\pa \A_{t,s}}{\pa t} e^{- \A_{t,s}^2}  ) - \frac{\pa}{\pa t}\tau _\sigma (\frac{\pa \A_{t,s}}{\pa s} e^{- \A_{t,s}^2}  )  \right],
$$
where $K_{t,s} = dt \wedge \frac{\pa \A_{t,s}}{\pa t} + ds \wedge \frac{\pa \A_{t,s}}{\pa s}$. Now we can compute
\begin{eqnarray*}
\int_0^1 \tau _\sigma (\frac{\pa \A_{t,1}}{\pa t} e^{- \A_{t,1}^2}  ) dt - \int_0^1 \tau _\sigma (\frac{\pa \A_{t,0}}{\pa t} e^{- \A_{t,0}^2}  ) dt & = & \int_0^1 \frac{\pa}{\pa s} \left[ \int_0^1 \tau _\sigma (\frac{\pa \A_{t,s}}{\pa t} e^{- \A_{t,s}^2}  ) dt \right] ds\\ 
& = & \int_{[0,1]^2} \frac{\pa}{\pa s} \tau _\sigma (\frac{\pa \A_{t,s}}{\pa t} e^{- \A_{t,s}^2}  ) dt \wedge ds\\ 
& = & \int_{[0,1]^2} \frac{\pa}{\pa t} \tau _\sigma (\frac{\pa \A_{t,s}}{\pa s} e^{- \A_{t,s}^2}  ) dt \wedge ds \\ & & + d_B \int_{[0,1]^2} \tau _\sigma (e^{-\A_{t,s}^2} \wedge K_{t,s}^2) \\ 
& = & \int_0^1 \left[ \tau _\sigma (  \frac{\pa \A_{1,s}}{\pa s} e^{- \A_{1,s}^2}) - \tau _\sigma (  \frac{\pa \A_{0,s}}{\pa s} e^{- \A_{0,s}^2} ) \right]  ds \\ & & + d_B \int_{[0,1]^2} \tau _\sigma (e^{-\A_{t,s}^2} \wedge K_{t,s}^2).
\end{eqnarray*}
Notice that
$$
\A_{1,s} = U^{-1} \A_{0,s} U \text{ and } \frac{\pa \A_{1,s}}{\pa s} = \B'_1 - \B_1 = U^{-1} \frac{\pa \A_{0,s}}{\pa s} U.
$$
Therefore, the proof is complete.
\end{proof}

The next proposition is an easy rephrasing of a result of Dai and Zhang:

\begin{proposition} \cite{DaiZhang}
Let  $\A$ be the Bismut superconnection associated with $\sigma D$, then the cohomology class of the differential form $\frac{-1}{\pi^{1/2}} \int_0^1 \tau _\sigma (\dot{\A_t} e^{-\A_t^2}) dt$ coincides with the Chern character of the spectral flow, i.e. 
$$
\ch ( \SF(D, U)) =  \frac{-1}{\pi^{1/2}} \left[\int_0^1 \tau _\sigma (\dot{\A_t} e^{-\A_t^2}) dt\right].
$$
\end{proposition}

The proof of this proposition relies on the good behaviour of the asymptotics of the rescaled Bismut superconnection. 
 To sum up, in order to prove 
 Theorem \ref{JLO=SF}, we are reduced to proving the following auxiliary result. 

\begin{theorem}\label{JLO=SF2}
When $\A=\B_\sigma$ and as differential forms on the base manifold $B$, we have the following equality:
$$
 \int_0^1 \tau _\sigma (\dot{\A_t} e^{-\A_t^2}) dt = \frac{1}{2} \sum_{k\geq 0} (-1)^k k!( \left<\left< U^{-1}, [\A, U],[\A, U^{-1}], \cdots, [\A, U]\right>\right>_{2k+1, \A}$$ 
$$- \langle\langle U, [\A, U^{-1}], [\A, U], \cdots , [\A, U], [\A, U^{-1}] \rangle\rangle_{2k+1,\A}) + \text{exact forms on the base}.
$$
\end{theorem}

Note that the additional exact forms can be given explicitly.

\subsection{Proof of the third theorem}

As explained before, the proof of this theorem follows the lines of \cite{GetzlerOdd} and we split the argument into a number of steps. First we double up our Hilbert space and  replace $U$ by 
$$
V :=\left(\begin{array}{cc} 0 & i U^{-1} \\ -i U & 0 \end{array} \right), \text{ so that } V^2 = I,
 \text{ and let }
 \tilde\B = \left( \begin{array}{cc} \B_\sigma & 0 \\ 0 & - \B_\sigma  \end{array}\right).$$
 Consider {{the operator $\A$ associated with the fibration 
$$
M\times [0,1] \times [0, +\infty) \rightarrow B\times [0,1]\times [0, +\infty ),
$$
and given by }}
$$
\A := {\tilde\B}_{t,x} + \left(\begin{array}{cc}d & 0 \\ 0 & d\end{array}\right)\text{ where } {\tilde\B}_{t,x} := \tilde\B_t + x V\text{ and } \tilde\B_t = \tilde\B - t V [{\tilde\B} , V].
$$
We use graded commutators so that for instance $[{\tilde\B} , V] = \tilde\B V + V \tilde\B$. The differential  $d$ is  the de Rham differential on $[0,1]\times [0, +\infty)$. {{ Moreover, we extend $\tau_\sigma$ to a supertrace $\tau_{s}$ given by
$$
\tau_{s} (A):= \tau _\sigma(A_{11}) +  \tau_\sigma (A_{22}).
$$}}
It is then straightforward to check{{, using the Bianchi identity satisfied by $\B_\sigma$, that  $d_B\tau_s(e^{-\A^2}) = 0$.}}  Computing the square of $\A$ one finds using for instance the relation $ V[\tilde\B, V] V = [\tilde\B, V]$:
$$
\A^2 = Y_{t,x} + dx V - dt V [{\tilde\B} , V]\text{ where }
Y_{t,x} =  ( \tilde\B_t )^2 + x (1-2t) [\tilde\B , V] + x^2.
$$
Recall then that the differential form $ 
\tau_{s}  (e^{-\A^2})$ is automatically closed as a differential form on the manifold with boundary $B\times [0,1]\times [0, +\infty)$.
\begin{lemma}
Let  $R(x_0)$ denote the rectangle $[0,1]\times [0, x_0]$, then in  $\Omega^*(B)$ we have:
$$
\int_{\pa R(x_0)} \tau_{s}  (e^{-\A^2}) \in d_B\Omega^*(B).
$$
\end{lemma}

\begin{proof}
We have by a direct computation 
$$
\int_{\pa R(x_0)} \tau_{s}  (e^{-\A^2}) = \int_{\pa R(x_0)} (dx \tau_{s} (V e^{- Y_{t,x}}) - dt \tau_{s} (  V [\tilde\B, V] e^{- Y_{t,x}})).
$$ 
Hence the differential form $\int_{\pa R(x_0)}\tau_{s}  (e^{-\A^2})$ on the base may be written as
$$
\int_{\pa R(x_0)}\tau_{s}  (e^{-\A^2}) = \int_{\pa R(x_0)}dx \omega_{t,x} - dt \alpha_{t,x},
$$
where $\omega_{t,x}$ and $\alpha_{t,x}$ are smooth families of differential forms on $B$. 
We thus have
\begin{eqnarray*}
 \int_{\partial R(x_0)} \tau_{s}  (e^{-\A^2}) & = & \int_0^1 [\alpha_{t,x_0} - \alpha_{t, 0}] dt - \int_0^{x_0} [\omega_{1,x} - \omega_{0,x}] dx\\ & = & \int_0^1 \int_0^{x_0} [\frac{\pa \alpha}{\pa x} - \frac{\pa \omega}{\pa t}] dt dx.
\end{eqnarray*}
The closedness of $\tau_{s}  (e^{- \A^2})$ implies in particular that the component
that contains $dt\wedge dx$, say $\beta_{t,x}dt\wedge dx$ 
satisfies  $$d_B \beta+\frac{\pa \alpha}{\pa x} - \frac{\pa \omega}{\pa t}=0.$$ But this is precisely what we need to complete the proof.
\end{proof}

We denote for $x>0$ by $\gamma_x$ the path $[0,1]\times \{x\}$ oriented in the direction of increasing $t\in [0,1]$. We also consider the path $\Gamma^x_t= \{t\}\times [0,x]$ for $t\in [0,1]$ and $x>0$, oriented in the direction of increasing $y\in [0,x]$. 

\begin{lemma}
 We have the following equality of the corresponding even forms
$$
\int_{\gamma_0} \tau_s (e^{- \A^2}) =  2 \int_0^1 \tau_{\sigma}  (\dot{\B_t} e^{-\B_t^2}) dt.
$$
\end{lemma}

\begin{proof}
We have 
$$
Y_{t,0} = \tilde\B _t^2 = \left(\begin{array}{cc} (\B_\sigma +t U^{-1} [\B_\sigma, U])^2 & 0 \\ 0 & (\B_\sigma + t U [\B_\sigma, U^{-1}])^2  \end{array}\right)
$$
Hence we obtain
\begin{multline*}
 \int_{\gamma_0} \tau_{s}  (e^{- \A^2}) = - \int_0^1 \tau_{s}  (V[{\tilde B}, V] e^{-Y_{t,0}}) dt = \int_0^1 \tau _\sigma ((U^{-1}[\B_\sigma, U]) e^{-(\B_\sigma + t U^{-1} [\B_\sigma, U])^2}) dt\\  - \int_0^1 \tau _\sigma (U[\B_\sigma, U^{-1}] e^{-(\B_\sigma +t U [\B_\sigma, U^{-1}])^2} )dt = 2 \int_0^1 \tau _\sigma (U^{-1}[\B_\sigma, U] e^{-(\B_\sigma +t U^{-1} [\B_\sigma, U])^2} )dt
\end{multline*}
where the last step is obained by $t\to 1-t$ in the second term.
{{Notice that only the even forms are relevant for us.}}
\end{proof}

\begin{lemma} We have 
 $$
\lim_{x\to +\infty} \left[ \int_{\Gamma_1^x} \tau_{s} (e^{-\A^2}) +  \int_{\Gamma_0^x} \tau_{s}  (e^{ -\A^2})\right]  = 0.
$$
Moreover up to forms that are exact on $B$, 
$$
\lim_{x\to +\infty} \left[ \int_{\Gamma_0^x} \tau_{s}  (e^{-\A^2})\right] = \frac{1}{2}(<\JLO (D), U> +<\JLO(D), U^{-1}>).
$$
\end{lemma}

\begin{proof}
Only the term $ - \tau_{s}  (V e^{-Y_{t,x}})$ contributes to the integrals $\int_{\Gamma_t^x}$. We thus need to compare $Y_{1,x}$ with $Y_{0,x}$. But notice that
$$
Y_{1,x} = \tilde\B_1^2 - x [\tilde\B , V] + x^2 = V \tilde\B^2 V - x[\tilde\B , V] + x^2.
$$
On the other hand, $V[\tilde\B, V] V = [\tilde\B , V]$ so that
$$
Y_{1,x} = V \left[ \tilde\B^2 - x [\tilde\B , V] + x^2 \right] V = V Y_{0, -x} V.
$$
Hence,
$$
 \tau_{s}  (V e^{-Y_{1,x}})  = \tau_{s}  (V^2 e^{-Y_{0,-x}} V) = \tau_{s} (e^{-Y_{0, -x}} V) = \tau_s  (V e^{-Y_{0, -x}}).
$$
 Therefore we obtain
$$
 \int_{\Gamma_1^x} \tau_{s}  (e^{-\A^2}) = \int_{-x}^0 \tau_{s}  (V e^{-Y_{0, y}}) dy.
$$
Now, since $Y_{0,x}=\tilde\B^2 + x[\tilde\B , V] + x^2$ and using Duhamel we know that
$$
\int_\R \tau_{s}  (V e^{Y_{0,x}}) dx = \sum_{k\geq 0} \left<\left<V, [\tilde\B, V], \cdots , [\tilde\B, V] \right>\right>_{\tilde\B} \int_\R  x^k e^{-x^2} dx,
$$
a series which converges in the Fr\'echet topology of $\Omega^*(B)$
because $\int_\R  x^k e^{-x^2}=\Gamma(\frac{k+1}{2})$ while the JLO bracket introduces a factor of $1/k!$ (see the proof of the next lemma for details). Next,  computing, 
$
\langle\langle V, [\tilde\B, V], \cdots , [\tilde\B, V] \rangle\rangle_{\tilde\B}
$ in terms of the multilinear functional corresponding to $\B_\sigma$, shows that it is trivial when the number of commutators $[\tilde\B, V]$ is even. Moreover, when $k=2\ell+1$ is odd clearly the integral $\int_\R  x^k e^{-x^2} dx$ vanishes, and so
$$
\int_\R \tau_{s}  (V e^{-Y_{0,x}}) dx = 0 \text{ or equivalently } \lim_{x\to +\infty} \int_{\Gamma_1^x} \tau_{s} (e^{-\A^2}) =  \lim_{x\to +\infty} \int_{\Gamma_0^x} \tau_{s} (e^{-\A^2}).
$$
If we integrate over $(0, +\infty)$ rather than $\R$ in the previous computation, then we obtain
\begin{multline*}
 \int_0^{+\infty} \tau_{s}  (V e^{-Y_{0,x}}) dx =  - \sum_{\ell\geq 0} \langle \langle V, [\tilde\B, V], \cdots , [\tilde\B, V] \rangle\rangle_{2\ell+1,\tilde\B} \int_0^\infty  x^{2\ell+1} e^{-x^2} dx \\ = -1/2 \sum_{\ell\geq 0} \ell! \langle\langle V, [\tilde\B, V], \cdots , [\tilde\B, V] \rangle\rangle_{2\ell+1,\tilde\B}.
\end{multline*}
Hence
\begin{multline*}
 V e^{-u_0\tilde\B^2} [\tilde\B , V] e^{-u_1\tilde\B^2} \cdots [\tilde\B , V] e^{-u_{2\ell+1}\tilde\B^2} = \\ (-1)^{\ell+1} \left(\begin{array}{cc}   U^{-1} e^{-u_0\B_\sigma^2} [\B_\sigma, U] e^{-u_1\B_\sigma^2}  \cdots [\B_\sigma, U] e^{-u_{2\ell+1}\B_\sigma^2} & 0 \\ 0 & - U e^{-u_0\B_\sigma^2} [\B_\sigma, U^{-1}] e^{-u_1\B_\sigma^2}  \cdots [\B_\sigma, U^{-1}] e^{-u_{2\ell+1}\B_\sigma^2}  \end{array}\right)
\end{multline*}
So that, using the fact that the differential forms involved are even, 
\begin{multline*}
(-1)^{\ell+1} \langle\langle V, [\tilde \B, V], \cdots , [\tilde\B, V] \rangle\rangle_{2\ell+1} = \langle\langle U^{-1}, [\B_\sigma, U], [\B_\sigma, U^{-1}], \cdots , [\B_\sigma, U^{-1}], [\B_\sigma, U] \rangle\rangle_{2\ell+1} \\ - \langle\langle U, [\B_\sigma, U^{-1}], [\B_\sigma, U], \cdots , [\B_\sigma, U], [\B_\sigma, U^{-1}] \rangle\rangle_{2\ell+1}.
\end{multline*}
\end{proof}

Now we remark that the last line of the proof can be simplified, up to the addition of exact forms on the base,
using the cocycle property of $\JLO(D)$
$$
\lim_{x\to +\infty} \int_{\Gamma_0^x} \tau_s (e^{-\A^2}) = \sum_{k\geq 0} (-1)^{k+1} k! \langle\langle U^{-1}, [\B_\sigma, U] , \cdots , [\B_\sigma, U^{-1}], [\B_\sigma, U] \rangle\rangle_{2k+1, \B_\sigma}.
$$
To end the proof of Theorem \ref{JLO=SF2}, we are reduced to the following

\begin{lemma}
In the Fr\'echet topology of $\Omega^*(B)$, we have:
$
\lim_{x_0\to +\infty} \int_{\gamma_{x_0}} \tau_s  (e^{-\A^2}) = 0.
$
\end{lemma}

\begin{proof}
Recall that
$$
\int_{\gamma_{x_0}} \tau_s  (e^{-\A^2}) = -\int_0^1 \tau_s  (V[\tilde\B, V] e^{-Y_{t,x_0}}) dt.
$$
Moreover, an application of our main theorem \ref{JLO} shows  that the following Duhamel expansion is convergent  in the Fr\'echet topology of $\Omega^*(B)$, with sum precisely $\tau_s  (V[\tilde\B, V] e^{-Y_{t,x_0}})$
$$
e^{-x_0^2} \sum_{k\geq 0} x_0^k (1-2t)^k \int_{\Delta(k)} \tau_s (V[\tilde\B , V] e^{-u_0\tilde\B_t^2} [\tilde\B , V] e^{-u_1\tilde\B_t^2} \cdots [\tilde\B , V] e^{-u_k\tilde\B_t^2} ) du_1 \cdots du_k.
$$
Reproducing the estimates of the semi-norms $(p_q)_{q\geq 0}$ of the expression
$$
\int_{\Delta(k)} \tau_s (V[\tilde\B , V] e^{-u_0\tilde\B_t^2} [\tilde\B , V] e^{-u_1\tilde\B_t^2} \cdots [\tilde\B , V] e^{-u_k\tilde\B_t^2} ) du_1 \cdots du_k
$$
we see that we can find constants $C_q$ depending on the unitary $U$ such that 
$$
p_q ( \int_{\Delta(k)} \tau_s  (V[\tilde\B , V] e^{-u_0\tilde\B_t^2} [\tilde\B , V] e^{-u_1\tilde\B_t^2} \cdots [\tilde\B , V] e^{-u_k\tilde\B_t^2} ) du_1 \cdots du_k) \leq C^k_q / k!.
$$ 
Finally notice that
$
\int_0^1 |1-2t|^k dt = 2/(k+1).
$
As a result we deduce
$$
p_q \left(\int_{\gamma_{x_0}} \tau_s  (e^{-\A^2})\right) \leq e^{-x_0^2} \sum_{k\geq 0} (C_q x_0)^k/(k+1)!,
$$
which converges to zero as $x_0\to +\infty$. 
\end{proof}


\begin{thebibliography}{9999}

\bibitem{AtiyahBook} Atiyah, M. F. {\em $K$-theory},  Lecture notes by D. W. Anderson W. A. Benjamin, Inc., New York-Amsterdam 1967.
\bibitem{AtiyahSinger4} Atiyah, M. F. and Singer, I. M. {\em The index of elliptic operators. IV.}  Ann. of Math. (2)  93  1971 119--138. 
\bibitem{BaumDouglas} Baum, P. and Douglas, R. G. {\em $K$ homology and index theory.}  Operator algebras and applications, Part I (Kingston, Ont., 1980),  pp. 117--173, Proc. Sympos. Pure Math., 38, Amer. Math. Soc., Providence, R.I., 1982.
\bibitem{BenameurGorokhovsky}  Benameur, M.-T. and Gorokhovsky, A., {\em Local index theorem for projective families. (English)}, arXiv:1007.3667v1 [math.DG], 
Fields Institute Communications 61, 1-27 (2011).
\bibitem{BHI}  Benameur, M.-T. and Heitsch, J. L., {\em Index theory and non-commutative geometry. I. Higher families index theory.}  $K$-Theory  33  (2004),  no. 2, 151--183. 
\bibitem{BH-JDG}  Benameur, M.-T. and Heitsch, J. L., {\em The twisted higher harmonic signature for foliations.} (English)
J. Differ. Geom. 87, No. 3, 389-468 (2011). 
\bibitem{BenameurPiazza}  Benameur, M.-T. and Piazza, P., {\em Index, eta and rho invariants on foliated bundles}, Ast\'erisque 327, 2009, 199-284.
\bibitem{BGV} Berline, N.; Getzler, E. and Vergne, M., {\em Heat kernels and Dirac operators.} Corrected reprint of the 1992 original. Grundlehren Text Editions. Springer-Verlag, Berlin, 2004.
\bibitem{Bismut} Bismut, J.-M., {\em The Atiyah-Singer index theorem for families of Dirac operators: two heat equation proofs}.  Invent. Math.  83  (1985),  no. 1, 91--151.
\bibitem{BF} J. Block, J. Fox,
{\em Asymptotic pseudodifferential operators and index yheory},
Contemp. Math., {\bf 105} (1990), 1--45. 
\bibitem{Bourbaki} Bourbaki, N., {\em Espaces vectoriels topologiques. Chapitres 1-5 (French) [Topological vector spaces. Chapters 1--5]},  \'El\'ements de math\'ematique. [Elements of mathematics] New edition. Masson, Paris, 1981. 
\bibitem{CPII} Carey, A. L. and Phillips, J., {\em Spectral flow in Fredholm modules, eta invariants and the JLO cocycle}. K-Theory, 31 (2004) 135-194.
 \bibitem{ConnesJLO} Connes, A., {\em Entire cyclic cohomology of Banach algebras and characters of $\theta$-summable Fredholm modules.}  $K$-Theory  1  (1988),  no. 6, 519--548.
\bibitem{ConnesBook} Connes, A., {\em Noncommutative geometry.} Academic Press, Inc., San Diego, CA, 1994.
\bibitem{DaiZhang} Dai, X. and Zhang, W., {\em Higher spectral flow.}  J. Funct. Anal.  157  (1998),  no. 2, 432--469.
\bibitem{GetzlerSzenes} Getzler, E. and  Szenesz, A., {\em On the Chern character of a theta-summable Fredholm module.}  J. Funct. Anal.  84  (1989),  no. 2, 343--357.
\bibitem{GetzlerOdd} Getzler, E., {\em The odd Chern character in cyclic homology and spectral flow.}  Topology  32  (1993),  no. 3, 489--507.
\bibitem{Gorokhovsky} Gorokhovsky. A., {\em Bivariant Chern character and longitudinal index.}  J. of Funct. Analysis 237  (2006),  105-134.
\bibitem{JLO} Jaffe, A.; Lesniewski, A. and Osterwalder, K. {\em Quantum $K$-theory. I. The Chern character.}  Comm. Math. Phys.  118  (1988),  no. 1, 1--14.
\bibitem{Kasparov} Kasparov, G. G., {\em Topological invariants of elliptic operators. I. $K$-homology.} (Russian)  Math. USSR-Izv.  9  (1975), no. 4, 751--792 (1976).;  translated from  Izv. Akad. Nauk SSSR Ser. Mat.  39  (1975), no. 4, 796--838(Russian)
\bibitem{Lance} E. C. Lance, {\em Hilbert $C^*$-Modules}, Cambridge University Press, Cambridge, 1995.
\bibitem{LP} E. Leichtnam and P. Piazza,  {\em Dirac index classes and the noncommutative spectral flow.} J. Funct. Anal.  200  (2003),  348--400. 
\bibitem{MelrosePiazza} Melrose, R. B. and Piazza, P., {\em Families of Dirac operators, boundaries and the $b$-calculus.}  J. Differential Geom.  46  (1997),  no. 1, 99--180.
\bibitem{MeyerThesis} Meyer, R., {\em Local and analytic cyclic homology.} EMS Tracts in Mathematics, 3. European Mathematical Society (EMS), Z\"urich, 2007.
\bibitem{Nistor} Nistor, V., {\em A bivariant Chern character for $p$-summable quasihomomorphisms.}  $K$-Theory  5  (1991),  no. 3, 193--211.
\bibitem{Perrot} Perrot, D., {\em A bivariant Chern character for families of spectral triples.}  Comm. Math. Phys.  231  (2002),  no. 1, 45--95.
\bibitem{Quillen88} Quillen, D., {\em Algebra cochains and cyclic cohomology.}  Inst. Hautes ¢Â?tudes Sci. Publ. Math.  No. 68  (1988), 139--174 (1989).
\bibitem{ReedSimon} Reed, M. and Simon, B.,  {\em Methods of modern mathematical physics. II. Fourier analysis}, self-adjointness. Academic Press [Harcourt Brace Jovanovich, Publishers], New York-London, 1975.
\bibitem{Vassout} Vassout, S. {\em Unbounded pseudodifferential calculus on Lie groupoids.}  J. Funct. Anal.  236  (2006),  no. 1, 161--200.
\bibitem{Wu97} Wu, F., {\em A bivariant Chern-Connes character and the higher $\Gamma$-index theorem.}  $K$-Theory  11  (1997),  no. 1, 35--82.
 \end{thebibliography}
\end{document}